\documentclass[11pt]{amsart} \textwidth=14.5cm \oddsidemargin=1cm
\usepackage[rgb]{xcolor}
\usepackage{tikz}
\usetikzlibrary{shapes,arrows}
\usepackage{verbatim}
\usepackage{amsmath}
\usepackage{amsxtra}
\usepackage{amscd}
\usepackage{amsthm}
\usepackage{amsfonts}
\usepackage{amssymb}
\usepackage{eucal}

\usepackage[linktocpage=true]{hyperref}
\usepackage[all, cmtip]{xy}
\usepackage{stmaryrd}
\usepackage{color}

\newtheorem{thm}{Theorem} [section]
\newtheorem{cor}[thm]{Corollary}
\newtheorem{lem}[thm]{Lemma}
\newtheorem{prop}[thm]{Proposition}

\theoremstyle{definition}
\newtheorem{definition}[thm]{Definition}

\theoremstyle{remark}
\newtheorem{rem}[thm]{Remark}

\numberwithin{equation}{section}

\begin{document}

\newcommand{\thmref}[1]{Theorem~\ref{#1}}
\newcommand{\secref}[1]{Section~\ref{#1}}
\newcommand{\lemref}[1]{Lemma~\ref{#1}}
\newcommand{\propref}[1]{Proposition~\ref{#1}}
\newcommand{\corref}[1]{Corollary~\ref{#1}}
\newcommand{\remref}[1]{Remark~\ref{#1}}
\newcommand{\eqnref}[1]{(\ref{#1})}
\newcommand{\exref}[1]{Example~\ref{#1}}

\newcommand{\nc}{\newcommand}
 \nc{\Z}{{\mathbb Z}}
 \nc{\C}{{\mathbb C}}
 \nc{\N}{{\mathbb N}}
 \nc{\F}{{\mf F}}
 \nc{\Q}{\mathbb{Q}}
 \nc{\la}{\lambda}
 \nc{\ep}{\epsilon}
 \nc{\h}{\mathfrak h}
 \nc{\n}{\mf n}
 \nc{\G}{{\mathfrak g}}
 \nc{\DG}{\widetilde{\mathfrak g}}
 \nc{\SG}{\breve{\mathfrak g}}
 \nc{\La}{\Lambda}
 \nc{\is}{{\mathbf i}}
 \nc{\V}{\mf V}
 \nc{\bi}{\bibitem}
 \nc{\E}{\mc E}
 \nc{\ba}{\tilde{\pa}}
 \nc{\half}{\frac{1}{2}}
 \nc{\hgt}{\text{ht}}
 \nc{\mc}{\mathcal}
 \nc{\mf}{\mathfrak} 
 \nc{\hf}{\frac{1}{2}}
 \nc{\hgl}{\widehat{\mathfrak{gl}}}
 \nc{\gl}{{\mathfrak{gl}}}
 \nc{\so}{{\mathfrak{so}}}
 \nc{\hz}{\hf+\Z}
 \nc{\ospO}{\mathcal{O}^{2m+1|2n}}
 \nc{\CatO}{\mathcal{O}}
\nc{\Omn}{\mc{O}^{m|n}}
\nc{\OO}{\mc{O}}

\nc{\U}{\bold{U}}
\nc{\SU}{\overline{\mathfrak u}}
\nc{\ov}{\overline}
\nc{\ul}{\underline}
\nc{\wt}{\widetilde}
\nc{\I}{\mathbb{I}}
\nc{\X}{\mathbb{X}}
\nc{\Y}{\mathbb{Y}}
\nc{\hh}{\widehat{\mf{h}}}
\nc{\aaa}{{\mf A}}
\nc{\xx}{{\mf x}}
\nc{\wty}{\widetilde{\mathbb Y}}
\nc{\ovy}{\overline{\mathbb Y}}
\nc{\vep}{\bar{\epsilon}}
\nc{\wotimes}{\widehat{\otimes}}

\newcommand{\wgv}{\Lambda^{\infty} \mathbb V}
\newcommand{\wgw}{\Lambda^{\infty} \mathbb V^*}

\nc{\ch}{\text{ch}}
\nc{\glmn}{\mf{gl}(m|n)}
\nc{\Uq}{{\mathcal U}_q}
\nc{\Uqgl}{\bold{U}_q(\gl_{2n+2})}
\nc{\VV}{\mathbb V}
\nc{\WW}{\mathbb W}
 \nc{\TL}{{\mathbb T}_{\mathbb L}}
 \nc{\TU}{{\mathbb T}_{\mathbb U}}
 \nc{\TLhat}{\widehat{\mathbb T}_{\mathbb L}}
 \nc{\TUhat}{\widehat{\mathbb T}_{\mathbb U}}
 \nc{\TLwt}{\widetilde{\mathbb T}_{\mathbb L}}
 \nc{\TUwt}{\widetilde{\mathbb T}_{\mathbb U}}
 \nc{\TLC}{\ddot{\mathbb T}_{\mathbb L}}
 \nc{\TUC}{\ddot{\mathbb T}_{\mathbb U}}

\nc{\osp}{\mathfrak{osp}}
 \nc{\be}{e}
 \nc{\bff}{f}
 \nc{\bk}{k}
 \nc{\bt}{t}
\nc{\bun}{{\bold{U}^{\imath}}}
\nc{\bundot}{{\dot{\bold{U}}^\imath}}
\nc{\id}{\text{id}}
\nc{\Ihf}{\I^\imath}
\nc{\Udot}{\dot{\U}}
\nc{\one}{\bold{1}}
\nc{\Uwedge}{\dot{\U}^{\wedge}}
\nc{\cun}{\Uj_q(\mathfrak{sl}_{2r+1}))}

\nc{\ospV}{\osp(2m+1|2n|\infty)}
\nc{\ospW}{\mathfrak{osp}(2m+1|2n+\infty)}
\nc{\epinftyV}{\epsilon^0_{\infty}}
\nc{\epinftyW}{\epsilon^1_{\infty}}
\nc{\wtlV}{X^{\ul \infty, +}_{{\bf b},0}}
\nc{\wtlW}{X^{\ul \infty, +}_{{\bf b},1}}
\nc{\wtlVk}{X^{\ul k, +}_{{\bf b},0}}
\nc{\wtlWk}{X^{\ul k, +}_{{\bf b},1}}
\nc{\f}{\bold{f}}
\nc{\fprime}{\bold{'f}}
\nc{\Qq}{\Q(q)}
\nc{\qq}{(q^{-1}-q)}
\nc{\uqsl}{\bold{U}_q({\mf{sl}_{2n+2}})}
\nc{\BLambda}{{\Lambda_{\inv}}}
\nc{\ThetaA}{\Theta}
\nc{\ThetaB}{\Theta^{\imath}}
\nc{\ThetaC}{\Theta^{\jmath}}
\nc{\B}{\bold{B}}
\nc{\Abar}{\psi}
\nc{\Bbar}{\psi_\imath}
\nc{\HBm}{\mc{H}_{B_m}}
\nc{\ua}{\mf{u}}
\nc{\nb}{u}
\nc{\wtA}{\texttt{wt}}
\nc{\bunlambda}{\Lambda^\imath}
\nc{\Dupsilon}{\Upsilon^{\vartriangle}}
\nc{\inv}{\theta}
\nc{\VVm}{\VV^{\otimes m}}
\nc{\mA}{\mathcal{A}}
\nc{\Tb}{{\mathbb{T}}^{\bf b}}
\nc{\wTb}{\widehat{\mathbb{T}}^{\bf b}}

\nc{\qqq}{(1-q^{-2})^{-1}}

\nc{\cundot}{\dot{\U}^{\jmath} }
\nc{\Iint}{\mathbb{I}^{\jmath} }
\nc{\ibe}[1]{e_{\alpha_{#1}}}
\nc{\ibff}[1]{f_{\alpha_{#1}}}
\nc{\ibk}[1]{k_{\alpha_{#1}}}
\nc{\wtC}{\texttt{wt}_{\jmath}}
\nc{\Cbar}{\psi_{\jmath}}

\newcommand{\blue}[1]{{\color{blue}#1}}
\newcommand{\red}[1]{{\color{red}#1}}
\newcommand{\green}[1]{{\color{green}#1}}
\newcommand{\white}[1]{{\color{white}#1}}

\newcommand{\Uj}{\U^{\jmath}}
\newcommand{\Ui}{\U^{\imath}}
\newcommand{\ibar}{\psi_{\imath}}

\title[Kazhdan-Lusztig Theory of super type D and quantum symmetric pairs]
{Kazhdan-Lusztig Theory of super type D and quantum symmetric pairs}
 
 \author[Huanchen Bao]{Huanchen Bao}
\address{Department of Mathematics, University of Maryland, College Park, MD 20742}
\email{huanchen@math.umd.edu}

\begin{abstract}

We reformulate the Kazhdan-Lusztig theory for the BGG category $\mathcal{O}$ of Lie algebras of type D via the theory of canonical bases arising from quantum symmetric pairs initiated by  Weiqiang Wang and the author. This is further applied to formulate and establish for the first time the  Kazhdan-Lusztig theory for the BGG category $\mathcal{O}$ of the ortho-symplectic Lie superalgebra $\mathfrak{osp}(2m|2n)$.

\end{abstract}

\maketitle

\let\thefootnote\relax\footnotetext{{\em 2010 Mathematics Subject Classification.} Primary 17B10.}

\setcounter{tocdepth}{1}
\tableofcontents

\section*{Introduction}

\subsection{}The Kazhdan-Lusztig theory provides a solution to the problem of determining the irreducible characters in the  BGG category $\mc{O}$ of semisimple Lie algebras (\cite{KL, BB, BK}). The theory was originally formulated in terms of the canonical bases (i.e., Kazhdan-Lusztig bases) of Hecke algebras. 
On the other hand, the classification of finite-dimensional simple Lie superalgebras over complex numbers has been obtained by Kac (\cite{Kac77}) in 1970's, while the representation theory of Lie superalgebras turns out to be very difficult. One of the main reasons is that the corresponding Weyl group of a Lie superalgebra is not enough to control the linkage principle in the BGG category $\mc{O}$. Thus the relevant Hecke algebras do not play  significant roles in the representation theory of Lie superalgebras as in the representation theory of semisimple Lie algebras.

 The Lie superalgebras $\mathfrak{gl}(m|n)$ and $\mathfrak{osp}(m|2n)$, which generalize the classical Lie algebras, are arguably the most important classes of Lie superalgebras. In 2003, Brundan in \cite{Br03} formulated a Kazhdan-Lusztig type conjecture for the full category $\mc{O}$ of general linear Lie superalgebras. The Jimbo-Schur (\cite{Jim}) duality plays a crucial role in Brundan's conjecture, which allows a reformulation of the Kazhdan-Lusztig theory in type A  in terms of the canonical bases of the quantum group  $\U_q(\mathfrak{sl}_{k})$ of type A.  Brundan's conjecture was proved first by  Cheng, Lam and Wang \cite{CLW15} and later by Brundan, Losev and Webster  \cite{BLW}. 
 
Recently in \cite{BW13}, Weiqiang Wang and the author initiated a theory of canonical bases arising from quantum symmetric pairs. We showed that a coideal subalgebra of $\U_q(\mathfrak{sl}_{k})$ centralizes the Hecke algebra of type B (of equal parameters) when acting on $\VV^{\otimes m}$, the tensor product of the natural representation $\VV$ of $\U_q(\mathfrak{sl}_{k})$. We constructed a (new) $\imath$-canonical basis on $\VV^{\otimes m}$, which allows a reformulation of the Kazhdan-Lusztig theory of type B independent of the Hecke algebra.
The theory was further applied to formulate and establish for the first time the Kazhdan-Lusztig theory for the BGG category $\mathcal{O}$ of the ortho-symplectic Lie superalgebra $\mathfrak{osp}(2m+1|2n)$. The geometric realization of the coideal subalgebras considered there and the canonical bases on the modified coideal subalgebras  have been given in \cite{BKLW} and \cite{LW15} using partial flag varieties of type B/C. 

On the other hand, the problem of determining the irreducible characters in the BGG category $\mc{O}$ of the ortho-symplectic Lie superalgebra $\mathfrak{osp}(2m|2n)$ is still  open  since 1970's.

\subsection{}
In this paper, we provide a complete solution to the irreducible character problem in the BGG category $\mc{O}$ of modules of integral and half-integral weights of the ortho-symplectic Lie superalgebra $\mathfrak{osp}(2m|2n)$. We adapt the theory of  canonical bases from \cite{BW13} to quantum symmetric pairs with different parameters. 
The non-super specialization here amounts a reformulation of the classical Kazhdan-Lusztig theory of type C/D. This paper is a  sequel of \cite{BW13}.

\subsection{}
A naive idea to follow \cite{BW13}  is to find the subalgebra  of $\U_q(\mathfrak{sl}_k)$, whose action on the tensor product $\VV^{\otimes m}$ centralizes the action of the Hecke algebra $\mc{H}_{D_m}$ of type D on $\VV^{\otimes m}$. Such (new) subalgebra has been constructed using the geometry of isotropic partial flag varieties of type D in \cite{FL14}. However the subalgebra is very involved, as expected, due to the complicated structure of isotropic flag varieties of type D, which makes it not suitable for further application to the category $\mc{O}$ of Lie superalgebras. 

We realize a natural and simple way to overcome the difficulty is to  first consider the Hecke algebra $\mathcal{H}^{1}_{B_m}$ of type B with unequal parameters. Let $\mathcal{H}^{p}_{B_m}$ be the Iwahori-Hecke algebra of type $B_m$ with two parameters $p$ and $q$ over $\mathbb Q(q, p)$, generated by $H^p_0, H_1, H_2, \dots , H_{m-1}$, and subject to certain relations (see \eqref{eq:HeckeB}). 
The Hecke algebra  $\mathcal{H}^{1}_{B_m}$ is the specialization of $\mathcal{H}^{p}_{B_m}$ at $p=1$.
We observe that $\mathcal{H}^{1}_{B_m}$ naturally contains the Hecke algebra   $\mc{H}_{D_m}$ of type D as a subalgebra. 

Then we look for the subalgebra  of $\U_q(\mathfrak{sl}_k)$, whose action on the tensor product $\VV^{\otimes m}$ centralizes the action of the Hecke algebra $\mathcal{H}^{1}_{B_m}$ on $\VV^{\otimes m}$. The subalgebra is a coideal subalgebra of the quantum group  $\U_q(\mathfrak{sl}_{k})$ of type A, denoted by $\U^{\imath}_q(\mathfrak{sl}_{k})$. Since the Hecke algebra $\mathcal{H}^{1}_{B_m}$ contains $\mc{H}_{D_m}$ as a subalgebra, the actions of $\U^{\imath}_q(\mathfrak{sl}_{k})$ and $\mc{H}_{D_m}$ on the tensor space $\VV^{\otimes n}$ clearly  commute. 

The coideal subalgebra comes in different forms depending on the parity of $k$. 
The quantum group $\U_q(\mathfrak{sl}_{k})$ and the coideal subalgebra $\U^{\imath}_q(\mathfrak{sl}_{k})$ form a quantum symmetric pair (\cite{Ko14}).

\subsection{}
Ehrig and Stroppel used the same coideal subalgebra $\Ui_q(\mathfrak{sl}_{k})$ of the quantum group $\U_q(\mathfrak{sl}_{k})$ to study the parabolic category $\mc{O}$ of the Lie algebra $\mathfrak{so}(2m)$ in \cite{ES13} simultaneously and independently from \cite{BW13}. 
They established the commutativity between the actions of $\Ui_q(\mathfrak{sl}_{k})$ and $\mathcal{H}_{D_m}$ on the tensor space $\VV^{\otimes m}$. The actions of the Chevalley type generators of $\Ui_q(\mathfrak{sl}_{k})$ on $\VV^{\otimes m}$ have been identified with the actions of translation functors on the category $\mc{O}$ of the Lie algebra $\mathfrak{so}(2m)$. However, neither the establishment of the $\imath$-canonical basis nor the (re)formulation of the (super) Kazhdan-Lusztig theory of type D was available there.

\subsection{}
It turns out the action of $\U^{\imath}_q(\mathfrak{sl}_{2r+1})$ (that is $k=2r+1$, being an odd number) on $(\VV^*)^{\otimes n}$, tensor product of the restricted dual $\VV^*$ of the natural representation $\VV$ of $\U_q(\mathfrak{sl}_{2r+1})$, centralizes the action of $\mc{H}^q_{B_n}$ ($p=q$) on $(\VV^*)^{\otimes n}$. It is actually more suitable to consider $\mc{H}^q_{B_n}$ as the Hecke algebra $\mc{H}_{C_n}$ of type C (of equal parameters) due to the connection with the BGG category $\mc{O}$ of the Lie superalgebra $\mathfrak{osp}(2m|2n)$ (in particular, the specialization $\mathfrak{sp}(2n)$ when $m =0$).


%

\subsection{}
	 In \cite{BW13} we considered the Hecke algebra $\mathcal{H}^{q}_{B_m}$ with equal parameters $p =q$. It was showed there  that certain coideal subalgebra of the quantum group $\U_q(\mathfrak{sl}_{k})$ of type A forms double centralizers with  $\mathcal{H}^{q}_{B_m}$ when acting on the tensor space $\VV^{\otimes m}$. In this paper, we consider the Hecke algebra $\mathcal{H}^{1}_{B_m}$ with $p =1$. Thus the (different) centralizing coideal subalgebra $\U^{\imath}_q(\mathfrak{sl}_{k})$ of $\U_q(\mathfrak{sl}_{k})$ considered in this paper is of different parameters from the one considered in \cite{BW13}, where the choice of the parameters in $\U^{\imath}_q(\mathfrak{sl}_{k})$ corresponds to the choice of the parameter $p$ in the two parameters Hecke algebra  $\mathcal{H}^{p}_{B_m}$. 
	 
The construction of $\imath$-canonical bases developed in \cite{BW13} applies to the coideal subalgebras with different parameters without difficulty, that is, simple $\U_q(\mathfrak{sl}_{k})$-modules and their tensor products admit $\imath$-canonical bases. In the ongoing work \cite{BW16}, we generalize the construction of $\imath$-canonical bases to more general quantum symmetric pairs (see also Remark~\ref{rem:BW16}).
	 
	 Thanks to the (weak) Schur type dualities for the type D (and type C), the classical Kazhdan-Lusztig theory of type D (and type C, respectively) can be reformulated in terms of $\imath$-canonical bases on $\VV^{\otimes m}$ (on $(\VV^*)^{\otimes n}$, respectively). More precisely speaking, the entries of the transition matrix between the $\imath$-canonical basis and the standard basis (i.e., the monomial basis) are exactly the  Kazhdan-Lusztig polynomials of type D (and type C, respectively.)

\subsection{}	 
We apply the theory of $\imath$-canonical bases to the BGG category $\mathcal{O}$ of the Lie superalgebra $\mathfrak{osp}(2m|2n)$  in Section~\ref{sec:rep}. In this section we consider the infinite rank limit of the quantum symmetric pair $(\U_q(\mathfrak{sl}_{\infty}), \Ui_q(\mathfrak{sl}_{\infty}))$. The theory of the super duality developed in \cite{CLW11} plays the essential role.

For a $0^m1^n$-sequence ${\bf b}$ (which consists of $m$ zeros and $n$ ones), we define a tensor space $\mathbb{T}^{\bf b}$ using $m$ copies of $\VV$ and $n$ copies of $\VV^*$ with the tensor order prescribed by ${\bf b}$ (with 0 corresponds to V). In this approach, $\mathbb{T}^{\bf b}$ (more precisely, its integral form) at $q = 1$ is identified with the Grothendieck group $[\mc{O}^{\bf b}]$ of the BGG category $\mc{O}^{\bf b}$ of $\mathfrak{osp}(2m|2n)$-modules (relative to a Borel subalgebra of type ${\bf b}$). We construct the $\imath$-canonical basis and dual $\imath$-canonical basis on  (a suitable completion of) the tensor space $\mathbb{T}^{\bf b}$. The construction of these bases is exactly the same as in \cite{BW13}, while only the precise formulas of these bases are different (which is irrelevant to the construction).

For the Lie superalgebra $\mathfrak{osp}(2m|2n)$, there are generally two types of fundamental systems, hence related Dynkin diagrams, with respect to different choices of the Borel subalgebras ${\bf b}$: with a type D branch (where ${\bf b}$ starts with $0^2$) in the Dynkin diagram; or with a type C branch (where ${\bf b}$ starts with $1$) in the Dynkin diagram. Those two types of fundamental systems are not conjugate under the Weyl group actions, but differ by odd reflections. 
We prove the Kazhdan-Lusztig theory for the category $\mc{O}^{\bf b}$ of the Lie superalgebra $\mathfrak{osp}(2m|2n)$  with respect to ${\bf b}$ of the first type by induction on $n$, where the base case $n=0$ follows from the classical Kazhdan-Lusztig theory of type D. On the other hand, the proof of the Kazhdan-Lusztig theory for the category $\mc{O}^{\bf b}$ of the Lie superalgebra $\mathfrak{osp}(2m|2n)$  with respect to ${\bf b}$ of the second type follows by induction on $m$, where the base case $m=0$ follows from the classical Kazhdan-Lusztig theory of type C. The induction processes of the two types are actually similar, where we compare category $\mc{O}$'s with respect to adjacent Borel subalgebras (switching adjacent $0$ and $1$ in the sequence ${\bf b}$), as well as compare the parabolic category $\mc{O}$ with the full category $\mc{O}$. We also study certain infinite rank limits of the parabolic category $\mc{O}$.

\vspace{.2cm}
\noindent {\bf Acknowledgement.}  The author would like to thank Weiqiang Wang for helpful discussion and suggestion on this paper. This research is partially supported by the AMS-Simons Travel Grant. A part of this paper was written when the author was visiting the Max-Planck-Institute in Bonn during summer 2015. He would like to thank the institute for its excellent working environment and support.

\section{Preliminaries on quantum groups}
 \label{sec:prelim}
 
 In this preliminary section, we review some basic definitions and constructions on quantum groups
 from Lusztig's book \cite{Lu94}. 
 We also introduce the involution $\inv$ and its quotient $\La_\inv$ which will be used in quantum symmetric pairs. 
 
\subsection{The involution $\inv$ and the lattice $\Lambda_\inv$}
  \label{subsec:theta}

Let $q$ be an indeterminate. 
For $r \in \N$, we define the following index sets:
\begin{align}
  \label{eq:I}
\begin{split}
\I_{2r+1} &= \{i \in \Z \mid -r \le i \le r\},  
 \\
\I_{2r}  &= \big\{i \in \Z+\hf \mid -r < i < r \big\}.
\end{split}
\end{align}

Set $k =2r+1$ or $2r$, and we use the shorthand notation $\I =\I_k$ in the remainder of Section~\ref{sec:prelim}, and very often throughout this paper.
Let  
\[
\Pi (= \Pi_k)= \big \{\alpha_i=\varepsilon_{i-\hf}-\varepsilon_{i+\hf} \mid i \in \I \big \}
\]
be the simple system of type $A_{k}$, and
let $\Phi$ be the associated root system. Denote by 
$$
\Lambda (= \Lambda_k) = \sum_{{i \in \I} } \big( \Z \varepsilon_{i - \hf} + \Z \varepsilon_{i + \hf} \big)
$$  
the integral weight lattice, and denote by $(\cdot, \cdot)$ the standard bilinear pairing on 
$\Lambda$ such that $(\varepsilon_a, 
\varepsilon_b) = \delta_{ab}$ for all $a,b$.  
For any $\mu = \sum_{i}c_i\alpha_i \in \N {\Pi}$, 
set $\hgt(\mu) = \sum_{i} c_i$. 

Let $\inv$ be the involution of the weight lattice $\Lambda$ such that 
\[
\inv(\varepsilon_{i-\hf}) = - \varepsilon_{-i+\hf}, \quad \text{ for all } i \in \I. 
\]
We shall also write $\lambda^{\inv} = \inv(\lambda)$, for $\lambda \in \Lambda$. The involution
$\inv$ preserves the bilinear form $(\cdot,\cdot)$ on the weight lattice $\Lambda$
and induces an automorphism on the simple system $\Pi$ such that $ \alpha^{\inv}_i = \alpha_{-i}$ for all $i \in \I$. 

Let $\Lambda^{\inv} = \{\mu \in \Lambda \mid \mu^{\inv} +\mu\}$ and 
$\Lambda_{\inv} = \Lambda/\Lambda^{\inv}$. For $\mu \in \Lambda$, denote by $\ov{\mu}$ the image of $\mu$ under the quotient map. 
There is a well-defined bilinear 
pairing  $\Z[\alpha_i - \alpha_{-i}]_{i \in \I} \times \BLambda  \rightarrow \Z$, such that 
$(\sum_{i > 0}a_i(\alpha_i-\alpha_{-i}), \ov{\mu}) :=   \sum_{i > 0}a_i(\alpha_i-\alpha_{-i}, \mu)$ for any $\ov{\mu} \in \BLambda$
with any preimage $\mu\in \Lambda$.

\subsection{The quantum group }

 The quantum group $\U = \U_q(\mf{sl}_{k+1})$ is defined to be the associative $\Q(q)$-algebra  
 generated by $E_{\alpha_i}$, $F_{\alpha_i}$, $K_{\alpha_i}$, $K^{-1}_{\alpha_i}$, $i \in \I$, subject to
the following relations (for $i$, $j \in \I$):
\begin{align*}
 K_{\alpha_i} K_{\alpha_i}^{-1} &= K_{\alpha_i}^{-1} K_{\alpha_i} =1,
  \\
 K_{\alpha_i} K_{\alpha_j} &= K_{\alpha_j} K_{\alpha_i},
  \\
 K_{\alpha_i} E_{\alpha_j} K_{\alpha_i}^{-1} &= q^{(\alpha_i, \alpha_j)} E_{\alpha_j}, \\\displaybreak[0]
 K_{\alpha_i} F_{\alpha_j} K_{\alpha_i}^{-1} &= q^{-(\alpha_i, \alpha_j)} F_{\alpha_j}, \\\displaybreak[0]
 E_{\alpha_i} F_{\alpha_j} -F_{\alpha_j} E_{\alpha_i} &= \delta_{i,j} \frac{K_{\alpha_i}
 -K^{-1}_{\alpha_i}}{q-q^{-1}}, \\ \displaybreak[0]
 E_{\alpha_i}^2 E_{\alpha_j} +E_{\alpha_j} E_{\alpha_i}^2 
 &= (q+q^{-1}) E_{\alpha_i} E_{\alpha_j} E_{\alpha_i},  \quad &&\text{if } |i-j|=1, \\
 E_{\alpha_i} E_{\alpha_j} &= E_{\alpha_j} E_{\alpha_i},  \,\qquad\qquad &&\text{if } |i-j|>1, \\
 F_{\alpha_i}^2 F_{\alpha_j} +F_{\alpha_j} F_{\alpha_i}^2 
 &= (q+q^{-1}) F_{\alpha_i} F_{\alpha_j} F_{\alpha_i},  \quad\, &&\text{if } |i-j|=1,\\
 F_{\alpha_i} F_{\alpha_j} &= F_{\alpha_j} F_{\alpha_i},  \qquad\ \qquad &&\text{if } |i-j|>1.
\end{align*}

Let $\U^+$, $\U^0$ and $\U^-$ be the $\Qq$-subalgebra of $\U$ generated by $E_{\alpha_i}$, $K^{\pm 1}_{\alpha_i}$, 
and $F_{\alpha_i}$  respectively, for  $i \in \I$. We introduce the divided power $F^{(a)}_{\alpha_i} = F^a_{\alpha_i}/[a]!$, where $a \ge 0$, 
$[a] = (q^{a}- q^{-a})/(q-q^{-1})$ and $[a]! = [1][2] \cdots [a]$. Let $\mA =\Z[q,q^{-1}]$.
Let $_\mA\U^{+}$ be the $\mA$-subalgebra of $\U^{+}$ 
generated by $E^{(a)}_{\alpha_i}$ for various $a \ge 0$ and $i \in \I$. Similarly let $_\mA\U^{-}$ be the $\mA$-subalgebra of $\U^{-}$ generated by $E^{(a)}_{\alpha_i}$ for various $a \ge 0$ and $i \in \I$.

\begin{prop}
\label{prop:invol}
\begin{enumerate}\label{prop:tauA}
\item There is an  involution $\omega$ on the $\Qq$-algebra $\U$
such that $\omega(E_{\alpha_i}) =F_{\alpha_i}$, $\omega(F_{\alpha_i}) =E_{\alpha_i}$, 
and $\omega(K_{\alpha_i}) = K^{-1}_{\alpha_i}$ for all $i \in \I$.

\item There is an anti-linear ($q \mapsto q^{-1}$) bar involution of the $\Q$-algebra $\U$
such that $\ov{E}_{\alpha_i}= E_{\alpha_i}$, $\ov{F}_{\alpha_i}=F_{\alpha_i}$, 
and $\ov{K}_{\alpha_i}=K_{\alpha_i}^{-1}$ for all $i \in \I$.

(Sometimes we denote the bar involution on $\U$ by $\psi$.)
%
\end{enumerate}
\end{prop}

Recall that $\U$ is a Hopf algebra with a coproduct 
%
\begin{align}  \label{eq:coprod}
\begin{split}
\Delta:  &\U \longrightarrow \U \otimes \U,
 \\
 \Delta (E_{\alpha_i}) &= 1 \otimes E_{\alpha_i} + E_{\alpha_i} \otimes K^{-1}_{\alpha_i}, \\
 \Delta (F_{\alpha_i}) &= F_{\alpha_i} \otimes 1 +  K_{\alpha_i} \otimes F_{\alpha_i},\\
 \Delta (K_{\alpha_i}) &= K_{\alpha_i} \otimes K_{\alpha_i}.
 \end{split}
\end{align}
There is a unique $\Qq$-algebra homomorphism
$\epsilon: \U \rightarrow \Qq$, called counit, 
such that $\epsilon(E_{\alpha_i}) = 0$, $\epsilon(F_{\alpha_i}) = 0$, and $\epsilon(K_{\alpha_i}) =1$.

\subsection{Braid group actions and canonical bases} 
\label{subsec:CB}

Let $W := W_{A_{{k}}} = \mf{S}_{{k}+1}$ be the Weyl group of type $A_{{k}}$. 
Recall \cite{Lu94} for each $\alpha_i$ and each finite-dimensional $\U$-module $M$, a linear operator 
$T_{\alpha_i}$ on $M$ is defined by,  for $\lambda \in \Lambda$ and $m \in M_\lambda$,
\[
T_{\alpha_i}(m) = \sum_{a, b, c \ge 0; -a+b-c
=(\lambda, \alpha_i)}(-1)^b q^{b-ac}E^{(a)}_{\alpha_i}F^{(b)}_{\alpha_i}E^{(c)}_{\alpha_i} m.
\]
These $T_{\alpha_i}$'s induce automorphisms of $\U$, denoted by $T_{\alpha_i}$ as well,  such that 
\[
T_{\alpha_i}(um) = T_{\alpha_i}(u)T_{\alpha_i}(m), \qquad \text{ for all } u \in \U, m \in M.
\] 
As automorphisms on $\U$ and as $\Qq$-linear isomorphisms on $M$,
the $T_{\alpha_i}$'s satisfy the braid group relations (\cite[Theorem 39.4.3]{Lu94}):
\begin{align*}
T_{\alpha_i}T_{\alpha_j} &= T_{\alpha_j}T_{\alpha_i}, &\text{ if } |i-j| >1 ,\\
T_{\alpha_i}T_{\alpha_j}T_{\alpha_i} &= T_{\alpha_j}T_{\alpha_i}T_{\alpha_j}, &\text{ if } |i-j| =1,
\end{align*}
Hence for each $w \in W$, $T_w$ can be defined independent of the choices of reduced expressions of $w$. 
(The $T_{\alpha_i}$ here is  $T''_{i,+}$ in \cite{Lu94}).

Denote by $\ell (\cdot)$ the length function of $W$, and let $w_0$ be the longest element of $W$. The following lemma is well-known (cf. \cite[Lemma~1.5]{BW13}).

\begin{lem}\label{lem:Tw0}
The following identities hold:
\[
T_{w_0}(K_{\alpha_i}) =K^{-1}_{\alpha_{-i}}, \quad T_{w_0}(E_{\alpha_i}) = -F_{\alpha_{-i}}K_{\alpha_{-i}}, \quad 
T_{w_0}(F_{\alpha_{-i}}) = - K^{-1}_{\alpha_i}E_{\alpha_i}, \quad \text{ for } i \in \I.
\]
\end{lem}
%
%

Let 
$\Lambda^+ = \{\lambda \in \Lambda \mid 2(\alpha_i , \lambda)/(\alpha_i,\alpha_i) \in {\N}, \forall i \in \I\}$
 be the set of dominant weights. Note that $\mu \in \Lambda^+$ if and only if $\mu ^{\inv} \in \Lambda^+$, 
 since the bilinear pairing $(\cdot,\cdot)$ on $ \Lambda$ is invariant under $\inv : \Lambda \rightarrow \Lambda$.

Let $M(\lambda)$ be the Verma module of $\U$ with highest weight $\lambda\in \Lambda$ and
with a highest weight vector denoted by $\eta$ or $\eta_{\lambda}$.  
We define a $\U$-module $^\omega M(\lambda)$, which 
has the same underlying vector space as $M(\lambda)$ but 
with the action twisted by the involution $\omega$ given in Proposition~\ref{prop:invol}.
When considering $\eta$ as a vector in $^\omega M(\lambda)$, 
we shall denote it by $\xi$ or $\xi_{-\lambda}$. 
The Verma module $M(\lambda)$ associated to dominant $\la \in \La^+$ has a unique finite-dimensional 
simple quotient $\U$-module, denoted by $L(\lambda)$. 
Similarly we define the $\U$-module $^\omega L(\lambda)$. 
For $\la \in \Lambda^+$, we let ${}_\mA L(\lambda) ={}_\mA\U^- \eta$ and $^\omega _\mA L(\lambda) ={}_\mA \U^+ \xi$ 
be the $\mA$-submodules of $L(\lambda)$ and $^\omega L(\lambda) $, respectively. 

We call a $\U$-module $M$ equipped with an anti-linear involution $\Abar$ {\em involutive} if
$
\Abar(u m) = \Abar(u) \Abar(m)$, $\forall u \in \U, m \in M$. The $\U$-modules ${^{\omega}L}(\lambda)$ and $L(\lambda)$ are both involutive. Given any two involutive $\U$-modules $M$ and $M'$, Lusztig showed that their tensor product $M \otimes M'$ is also involutive (\cite[\S27.31]{Lu94}).

In \cite{Lu90, Lu94} and \cite{Ka}, the canonical basis $\bold{B}$ of ${_\mA \U^+} \cong {_\mA \U^-}$ has been constructed. 
For any element $b \in \B$, when considered as an element in $\U^-$ or $\U^+$, we shall denote it by $b^-$ or $b^+$,  respectively. 
In \cite{Lu94}, subsets $\B(\lambda)$ of $\B$ is also constructed for each $\lambda \in \Lambda^+$, 
such that $\{b^-\eta_{\lambda} \mid b \in \B(\lambda)\}$ gives the canonical basis of ${}_\mA L(\lambda)$. 
Similarly $\{b^+ \xi_{-\lambda} \mid b \in \B(\lambda)\}$ gives the canonical basis of ${^{\omega}_\mA L(\lambda)}$.  
\section{Quantum Symmetric pairs} \label{sec:QSP}
In this section we shall develop the theory of $\imath$-canonical bases for the quantum symmetric pairs :
\[
(\U_q(\mathfrak{sl}_{2r+1}), \Ui_q(\mathfrak{sl}_{2r+1}))\quad \text{ and } \quad (\U_q(\mathfrak{sl}_{2r+2}), \Ui_q(\mathfrak{sl}_{2r+2})).
\]

The definitions of the quantum symmetric pairs shall be given in the first two sections, separately (see also \cite{ES13}). The theory of $\imath$-canonical bases is nevertheless uniform in both cases. Therefore after stating their definitions we shall formulate their general theory together. Section $2.1$ - Section $2.3$ are analogous to \cite[Part~1]{BW13} hence we shall omit the proofs almost entirely.  Some of the quantum symmetric paris considered here are of different parameters than the ones considered in \cite{BW13} (See also \cite{BW16} for more general construction). We refer to \cite{Ko14} for general theory of quantum symmetric pairs.

\subsection{The quantum symmetric pair $(\U_q(\mathfrak{sl}_{2r+1}), \Ui_q(\mathfrak{sl}_{2r+1}))$}\label{sec:Ui}
We define 
$$
\I^{\imath} _{2r} = (\hf +\N) \cap \I_{2r} 
= \Big\{\hf, \frac32, \ldots, r-\hf \Big\}.
$$ 
The Dynkin diagram of type $A_{2r}$ together with the involution $\inv$ is depicted as follows:
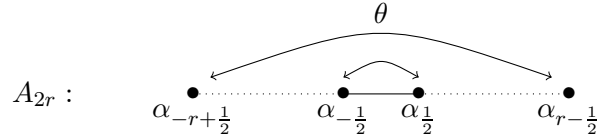
\begin{center}
\begin{tikzpicture}
\draw (-1.5,0) node {$A_{2r}:$};
 \draw[dotted]  (0.5,0) node[below]  {$\alpha_{-r+\hf}$} -- (2.5,0) node[below]  {$\alpha_{-\hf}$} ;
 \draw (2.5,0)
 -- (3.5,0) node[below]  {$\alpha_{\hf}$};
 \draw[dotted] (3.5,0) -- (5.5,0) node[below] {$\alpha_{r-\hf}$} ;
\draw (0.5,0) node (-r) {$\bullet$};
 \draw (2.5,0) node (-1) {$\bullet$};
\draw (3.5,0) node (1) {$\bullet$}; 
\draw (5.5,0) node (r) {$\bullet$};
\draw[<->] (-r.north east) .. controls (3,1) .. node[above] {$\theta$} (r.north west) ;
\draw[<->] (-1.north) .. controls (3,0.5) ..  (1.north) ;
\end{tikzpicture}
\end{center}

The algebra $\Ui_q(\mathfrak{sl}_{2r+1})$ is defined to be the associative algebra over $\Q(q)$ 
generated by $\be_{\alpha_i}$, $\bff_{\alpha_i}$, $\bk_{\alpha_i}$, $\bk^{-1}_{\alpha_i}$, $i \in \I^{\imath} _{2r} $,
subject to the following relations for $i, j \in \I^{\imath} _{2r} $:
\begin{align}
 \ibk{i} \ibk{i}^{-1} &= \ibk{i}^{-1} \ibk{i} =1, \displaybreak[0] \notag\\
 \ibk{i} \ibk{j} &=  \ibk{j}  \ibk{i}, \displaybreak[0] \notag\\
 \ibk{i} \ibe{j} \ibk{i}^{-1} &= q^{(\alpha_i-\alpha_{-i},\alpha_j)} \ibe{j}, \displaybreak[0] \notag\\
 \ibk{i} \ibff{j} \ibk{i}^{-1} &= q^{-(\alpha_i-\alpha_{-i},\alpha_j)}\ibff{j}, \displaybreak[0] \notag\\
 \be_{\alpha_i} \ibff{j} -\ibff{i} \ibe{j} &= \delta_{i,j} \frac{\bk_{\alpha_i}
 -\bk^{-1}_{\alpha_i}}{q-q^{-1}},         \qquad\; \qquad \text{if } i, j \neq \hf, \displaybreak[0] \notag\\
 \ibe{i}^2 \ibe{j} +\ibe{j} \ibe{i}^2 &= (q+q^{-1}) \ibe{i} \ibe{j} \ibe{i},  \qquad \text{if }  |i-j|=1, \displaybreak[0] \notag \\
 \ibff{i}^2 \ibff{j} +\ibff{j} \ibff{i}^2 &= (q+q^{-1}) \ibff{i} \ibff{j} \ibff{i}, \qquad \text{if } |i-j|=1,\displaybreak[0]\notag \\
 \ibe{i} \ibe{j} &= \ibe{j} \ibe{i},     \quad\qquad\qquad\qquad\; \text{if } |i-j|>1, \displaybreak[0]\notag \\
 \ibff{i} \ibff{j}  &=\ibff{j}  \ibff{i},   \quad\qquad\qquad\qquad\; \text{if } |i-j|>1,\displaybreak[0]\notag \\
 %
 %
  \ibff{\hf}^2\ibe{\hf} + \ibe{\hf}\ibff{\hf}^2
    &= (q+q^{-1}) \Big(\ibff{\hf}\ibe{\hf}\ibff{\hf}-q^2\ibff{\hf}\ibk{\hf}^{-1}-q^{-2}\ibff{\hf}\ibk{\hf} \big),\displaybreak[0] \label{eq:Serre:1}\\
 \ibe{\hf}^2\ibff{\hf} + \ibff{\hf}\ibe{\hf}^2
   &= (q+q^{-1}) \Big(\ibe{\hf}\ibff{\hf}\ibe{\hf}-q^{-2}\ibk{\hf}\ibe{\hf} -q^2\ibk{\hf}^{-1}\ibe{\hf} \Big).\displaybreak[0] \label{eq:Serre:2}
\end{align}
We introduce the divided powers 
$\be^{(a)}_{\alpha_i} = \be^a_{\alpha_i} / [a]!$, $\bff^{(a)}_{\alpha_i} = \bff^a_{\alpha_i} / [a]!$.

\begin{rem}
Note that the last two ``Serre" type relations \eqref{eq:Serre:1} and  \eqref{eq:Serre:2} are different from \cite[\S6.1]{BW13}.
\end{rem}

\begin{lem} 
 The algebra $\Ui_q(\mathfrak{sl}_{2r+1})$ has an anti-linear ($q \mapsto q^{-1}$) bar involution such that 
 $\ov{\bk}_{\alpha_i} = \bk^{-1}_{\alpha_i}$, $\ov{\be}_{\alpha_i} = \be_{\alpha_i}$, 
 and $\ov{\bff}_{\alpha_i} = \bff_{\alpha_i}$,  for all $i \in \I^{\imath} _{2r} $. 

(Sometimes we denote the bar involution on $\Ui_q(\mathfrak{sl}_{2r+1}))$ by $\ibar$.)

\end{lem}


\begin{prop} \label{int:prop:embedding} 
There is an injective $\Qq$-algebra homomorphism $\imath :  \Ui_q(\mathfrak{sl}_{2r+1}) \rightarrow \U_q(\mathfrak{sl}_{2r+1})$ defined by, for all  $i \in \I^{\imath} _{2r}$,
 \begin{align*}
\bk_{\alpha_i} \mapsto K_{\alpha_i}K^{-1}_{\alpha_{-i}}, \quad
\be_{\alpha_i} \mapsto  E_{\alpha_i} +  F_{\alpha_{-i}}K^{-1}_{\alpha_i}, \quad
\bff_{\alpha_i} \mapsto K^{-1}_{\alpha_{-i}} F_{\alpha_i} + E_{\alpha_{-i}}.
\end{align*}
\end{prop}

\begin{rem}
In \cite[Proposition~2.2]{BW13}, we consider the embedding that maps $\be_{\alpha_i}$ to $E_{\alpha_i} +  q^{-\delta_{i, \hf}} F_{\alpha_{-i}}K^{-1}_{\alpha_i}$ and maps $\bff_{\alpha_i}$ to $q^{\delta_{i, \hf}} K^{-1}_{\alpha_{-i}} F_{\alpha_i} + E_{\alpha_{-i}}$. 
\end{rem}
%

Note that $E_{\alpha_i} (K^{-1}_{\alpha_i}F_{\alpha_{-i}}) 
= q^{2} (K^{-1}_{\alpha_i}F_{\alpha_{-i}}) E_{\alpha_i}$ for $i \in \I^{\imath} _{2r} $. 
We have for $i \in \I^{\imath} _{2r} $, 
\begin{align}
\label{int:eq:beZ} \imath(\be^{(a)}_{\alpha_i}) 
 &= \sum^{a}_{j=0}  q^{j(a-j)}{(F_{\alpha_{-i}} K^{-1}_{\alpha_i} )^j \over [j]!}\frac{E^{a-j}_{\alpha_i}}{[a-j]!},
   \\
\label{int:eq:bffZ} \imath(\bff^{(a)}_{\alpha_i}) 
 &= \sum^{a}_{j=0}  q^{j(a-j)}{(K^{-1}_{\alpha_{-i}}F_{\alpha_{i}})^j \over [j]!}\frac{E^{a-j}_{\alpha_{-i}}}{[a-j]!}.
\end{align}

\begin{prop} \label{int:prop:coproduct}
The coproduct $\Delta$ on $\U_q(\mathfrak{sl}_{2r+1})$ restricts under the embedding $\imath$ to
a $\Qq$-algebra homomorphism 
\[
\Delta : \Ui_q(\mathfrak{sl}_{2r+1}) \longrightarrow \Ui_q(\mathfrak{sl}_{2r+1}) \otimes \U_q(\mathfrak{sl}_{2r+1})
\]
such that for all $i \in \I^{\imath} _{2r} $, 
\begin{align*}
\Delta(\bk_{\alpha_i}) &= \bk_{\alpha_i} \otimes K_{\alpha_i} K^{-1}_{\alpha_{-i}},
 \\
\Delta({\be_{\alpha_i}}) &= 1 \otimes E_{\alpha_i} + \be_{\alpha_i} \otimes 
 K^{-1}_{\alpha_i} + \bk^{-1}_{\alpha_i} \otimes  F_{\alpha_{-i}} K^{-1}_{\alpha_i} ,
  \\
\Delta (\bff_{\alpha_i}) &= \bk_{\alpha_i} \otimes K^{-1}_{\alpha_{-i}} F_{\alpha_i}
  + \bff_{\alpha_i} \otimes K^{-1}_{\alpha_{-i}} + 1 \otimes E_{{\alpha_{-i}}}.
\end{align*}
Similarly, the counit $\epsilon$ of $\U_q(\mathfrak{sl}_{2r+1})$ induces a $\Qq$-algebra homomorphism  
\[
\epsilon : \Ui_q(\mathfrak{sl}_{2r+1}) \rightarrow \Qq
\]
such that 
$\epsilon(\be_{\alpha_i}) =\epsilon(\bff_{\alpha_i})=0$ and $\epsilon(\bk_{\alpha_i}) =1$ for all $i \in \I^{\imath} _{2r} $.
\end{prop}

It follows by Proposition~\ref{int:prop:coproduct}
that $\Ui_q(\mathfrak{sl}_{2r+1})$  is a  (right) coideal subalgebra of $\U$. 
The map $\Delta : \Ui_q(\mathfrak{sl}_{2r+1}) \rightarrow \Ui_q(\mathfrak{sl}_{2r+1}) \otimes \U_q(\mathfrak{sl}_{2r+1})$ will be called the coproduct of $\Ui_q(\mathfrak{sl}_{2r+1})$ 
and $ \epsilon : \Ui_q(\mathfrak{sl}_{2r+1}) \rightarrow \Qq$ will be called the  counit of $\Ui_q(\mathfrak{sl}_{2r+1})$.  
The coproduct $\Delta : \Ui_q(\mathfrak{sl}_{2r+1}) \rightarrow \Ui_q(\mathfrak{sl}_{2r+1}) \otimes \U_q(\mathfrak{sl}_{2r+1})$  is coassociative, i.e.,
$
(1 \otimes \Delta) \Delta = (\Delta \otimes 1)\Delta: \Ui_q(\mathfrak{sl}_{2r+1}) \rightarrow \Ui_q(\mathfrak{sl}_{2r+1}) \otimes \U_q(\mathfrak{sl}_{2r+1}) \otimes \U_q(\mathfrak{sl}_{2r+1})$.
The counit map $\epsilon$ makes $\Qq$ a (trivial) $\Ui_q(\mathfrak{sl}_{2r+1})$-module.  
Let $m : \U_q(\mathfrak{sl}_{2r+1}) \otimes \U_q(\mathfrak{sl}_{2r+1}) \rightarrow \U_q(\mathfrak{sl}_{2r+1})$ denote the multiplication map. 
We have 
$
m (\epsilon \otimes 1)\Delta = \imath : \Ui_q(\mathfrak{sl}_{2r+1}) \longrightarrow \U_q(\mathfrak{sl}_{2r+1})
$ by direct computation.


\subsection{The quantum symmetric pair $(\U_q(\mathfrak{sl}_{2r+2}), \Ui_q(\mathfrak{sl}_{2r+2}))$}\label{sec:Uj}
We set
\begin{equation}  \label{eq:Ihf}
\I^{\imath}_{2r+1}  :=\Z_{>0} \cap \I_{2r+1} =\{1, \ldots, r\}.
\end{equation}

The Dynkin diagram of type $A_{2r+1}$ together with the involution $\inv$
can be depicted as follows: 

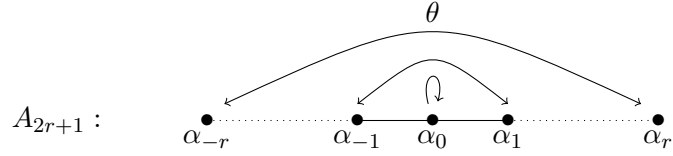
\begin{center}
\begin{tikzpicture}

\draw (-2,0) node {$A_{2r+1}:$};
 \draw[dotted]  (0,0) node[below] {$\alpha_{-r}$} -- (2,0) node[below] {$\alpha_{-1}$} ;
 \draw (2,0) -- (3,0) node[below]  {$\alpha_{0}$}
 -- (4,0) node[below] {$\alpha_{1}$};
 \draw[dotted] (4,0) -- (6,0) node[below] {$\alpha_{r}$} ;
\draw (0,0) node (-r) {$\bullet$};
 \draw (2,0) node (-1) {$\bullet$};
\draw (3,0) node (0) {$\bullet$};
\draw (4,0) node (1) {$\bullet$}; 
\draw (6,0) node (r) {$\bullet$};
\draw[<->] (-r.north east) .. controls (3,1.5) .. node[above] {$\theta$} (r.north west) ;
\draw[<->] (-1.north) .. controls (3,1) ..  (1.north) ;
\draw[<->] (0) edge[<->, loop above] (0);
\end{tikzpicture}
\end{center}

The algebra $\Ui_q(\mathfrak{sl}_{2r+2})$ is defined to be the associative algebra over $\Q(q)$ generated by  
$\be_{\alpha_i}$, $\bff_{\alpha_i}$, $\bk_{\alpha_i}$, $\bk^{-1}_{\alpha_i}$ ($i \in \I^{\imath}_{2r+1}$) , and $\bt$, 
subject to the following relations for $i$, $j \in \I^{\imath}_{2r+1}$:
\begin{align*}
 \bk_{\alpha_i} \bk_{\alpha_i}^{-1} &= \bk_{\alpha_i}^{-1} \bk_{\alpha_i} =1,\displaybreak[0]\\
 \bk_{\alpha_i} \bk_{\alpha_j} &= \bk_{\alpha_j} \bk_{\alpha_i}, \displaybreak[0]\\
 \bk_{\alpha_i} \be_{\alpha_j} \bk_{\alpha_i}^{-1} &= q^{(\alpha_i-\alpha_{-i}, \alpha_j)} \be_{\alpha_j}, \displaybreak[0]\\
 \bk_{\alpha_i} \bff_{\alpha_j} \bk_{\alpha_i}^{-1} &= q^{-(\alpha_i-\alpha_{-i}, \alpha_j)}
 \bff_{\alpha_j}, \displaybreak[0]\\
 \bk_{\alpha_i}\bt\bk^{-1}_{\alpha_i} &= \bt, \displaybreak[0]\\
 \be_{\alpha_i} \bff_{\alpha_j} -\bff_{\alpha_j} \be_{\alpha_i} &= \delta_{i,j} \frac{\bk_{\alpha_i}
 -\bk^{-1}_{\alpha_i}}{q-q^{-1}}, \displaybreak[0]\\
 \be_{\alpha_i}^2 \be_{\alpha_j} +\be_{\alpha_j} \be_{\alpha_i}^2 &= (q+q^{-1}) \be_{\alpha_i} \be_{\alpha_j} \be_{\alpha_i}, 
   \ \ \quad\quad &\text{if }& |i-j|=1, \displaybreak[0]\\
 \be_{\alpha_i} \be_{\alpha_j} &= \be_{\alpha_j} \be_{\alpha_i}, \ \qquad\qquad\ \ \ \ \qquad &\text{if }& |i-j|>1, \displaybreak[0]\\
 \bff_{\alpha_i}^2 \bff_{\alpha_j} +\bff_{\alpha_j} \bff_{\alpha_i}^2 &= (q+q^{-1}) \bff_{\alpha_i} \bff_{\alpha_j} \bff_{\alpha_i},
    \ \ \quad\quad &\text{if }& |i-j|=1,\displaybreak[0]\\
 \bff_{\alpha_i} \bff_{\alpha_j}  &= \bff_{\alpha_j}  \bff_{\alpha_i},  \ \qquad\qquad\ \ \ \ \qquad &\text{if }& |i-j|>1, \displaybreak[0]\\
 \be_{\alpha_i}\bt &=\bt\be_{\alpha_i}, \quad\qquad\quad &\text{if }&  i > 1, \displaybreak[0]\\
 \be_{\alpha_1}^2\bt + \bt\be_{\alpha_1}^2 &= (q+q^{-1}) \be_{\alpha_1}\bt\be_{\alpha_1},\displaybreak[0]\\
 \bt^2\be_{\alpha_1} + \be_{\alpha_1}\bt^2 &= (q + q^{-1}) \bt\be_{\alpha_1}\bt + \be_{\alpha_1},\displaybreak[0]\\
 \bff_{\alpha_i}\bt &=\bt\bff_{\alpha_i}, \quad\qquad&\text{if }&  i > 1, \displaybreak[0]\\
 \bff_{\alpha_1}^2\bt + \bt\bff_{\alpha_1}^2 &= (q+q^{-1}) \bff_{\alpha_1}\bt\bff_{\alpha_1},\displaybreak[0]\\
 \bt^2\bff_{\alpha_1} + \bff_{\alpha_1}\bt^2 &= (q + q^{-1}) \bt\bff_{\alpha_1}\bt + \bff_{\alpha_1}.\displaybreak[0]
\end{align*}
We introduce the divided powers $\be^{(a)}_{\alpha_i} = \be^a_{i} / [a]!$, $\bff^{(a)}_{\alpha_i} = \bff^a_{i} / [a]!$ for  $a \ge 0$, $i \in \I^{\imath}_{2r+1}$.

\begin{rem}\label{rem:Uisame}
The generating relations of the algebra $\Ui_q(\mathfrak{sl}_{2r+2})$ are the same as the generating relations of the algebra considered in \cite[\S2.1]{BW13}.
\end{rem}

\begin{lem}   \label{lem:3inv}
The $\Q$-algebra $\Ui_q(\mathfrak{sl}_{2r+2})$ has an anti-linear ($q \mapsto q^{-1}$) bar involution such that 
$\ov{\bk}_{\alpha_i} = \bk^{-1}_{\alpha_i}$, 
$\ov{\be}_{\alpha_i} = \be_{\alpha_i}$, $\ov{\bff}_{\alpha_i} = \bff_{\alpha_i}$, and  $\ov{\bt} = \bt$ for all $i \in \I^{\imath}_{2r+1}$. 

(Sometimes we denote the bar involution on $\Ui_q(\mathfrak{sl}_{2r+2})$ by $\Bbar$.)
\end{lem}

\begin{prop} \label{prop:embedding}  
There is an injective $\Qq$-algebra homomorphism $\imath :  \Ui_q(\mathfrak{sl}_{2r+2}) \rightarrow \U_q(\mathfrak{sl}_{2r+2})$
which sends (for $i \in \I^{\imath}_{2r+1}$)
 \begin{align*}
\bk_{\alpha_i} \mapsto K_{\alpha_i}K^{-1}_{\alpha_{-i}}, \qquad & \bt \mapsto E_{\alpha_0} +qF_{\alpha_0}K^{-1}_{\alpha_0} \\
 \be_{\alpha_i} \mapsto  E_{\alpha_i} + K^{-1}_{\alpha_i}F_{\alpha_{-i}},\qquad 
& \bff_{\alpha_i} \mapsto F_{\alpha_i} K^{-1}_{\alpha_{-i}}+ E_{\alpha_{-i}}.
\end{align*}
\end{prop}

\begin{rem}
The embedding $\imath$ in Proposition~\ref{prop:embedding} is different from the embedding in \cite[Proposition~2.2]{BW13}, although the two subalgebras are (abstractly) isomorphic (see Remark~\ref{rem:Uisame}). This phenomenon for quantum symmetric pairs was first observed in \cite[Section~5]{Le}. 
\end{rem}

Note that $E_{\alpha_i} (K^{-1}_{\alpha_i}F_{\alpha_{-i}}) = q^{2} (K^{-1}_{\alpha_i}F_{\alpha_{-i}}) E_{\alpha_i}$ 
for all $ 0 \neq i \in \I$. Using the quantum binomial formula \cite[1.3.5]{Lu94},  we have, for all $i \in \I^{\imath}_{2r+1}$, $a \in \N$,
\begin{align}
\label{eq:beZ}
\imath(\be^{(a)}_{\alpha_i}) &= \sum^{a}_{j=0}  q^{j(a-j)}F^{(j)}_{\alpha_{-i}}K^{-j}_{\alpha_i} E^{(a-j)}_{\alpha_i},
 \\
\label{eq:bffZ}
\imath(\bff^{(a)}_{\alpha_i}) &= \sum^{a}_{j=0}  q^{j(a-j)}F^{(j)}_{\alpha_{i}}K^{-j}_{\alpha_{-i}} E^{(a-j)}_{\alpha_{-i}}.
\end{align}

\begin{prop} \label{prop:coproduct}
The coproduct $\Delta$ on $\U_q(\mathfrak{sl}_{2r+2})$ restricts via the embedding $\imath$ to a $\Qq$-algebra homomorphism 
\[
\Delta : \Ui_q(\mathfrak{sl}_{2r+2}) \longrightarrow \Ui_q(\mathfrak{sl}_{2r+2}) \otimes \U_q(\mathfrak{sl}_{2r+2})
\]
such that, for all $i \in \I^{\imath}_{2r+1}$, 
\begin{align*}
\Delta(\bk_{\alpha_i}) &= \bk_{\alpha_i} \otimes K_{\alpha_i} K^{-1}_{\alpha_{-i}},
 \\
\Delta({\be_{\alpha_i}}) &= 1 \otimes E_{\alpha_i} + \be_{\alpha_i} \otimes K^{-1}_{\alpha_i} 
+ \bk^{-1}_{\alpha_i} \otimes K^{-1}_{\alpha_i}F_{\alpha_{-i}},
 \\
\Delta (\bff_{\alpha_i}) &= \bk_{\alpha_i} \otimes F_{\alpha_i}K^{-1}_{\alpha_{-i}} + \bff_{\alpha_i} 
\otimes K^{-1}_{\alpha_{-i}} + 1 \otimes E_{{\alpha_{-i}}},
 \\
\Delta(\bt) &= \bt \otimes K^{-1}_{\alpha_0} + 1 \otimes q F_{\alpha_0}K^{-1}_{\alpha_0}+ 1 \otimes E_{\alpha_0}.
\end{align*}
Similarly, the counit $\epsilon$ of $\U_q(\mathfrak{sl}_{2r+2})$ induces a $\Qq$-algebra homomorphism 
\[
 \epsilon : \Ui_q(\mathfrak{sl}_{2r+2}) \rightarrow \Qq
 \]
 such that 
$\epsilon(\be_{\alpha_i}) =\epsilon(\bff_{\alpha_i})=0$, $\epsilon(\bt) = 0$, and $\epsilon(\bk_{\alpha_i}) =1$ for all $i \in \I^{\imath}_{2r+1}$.
\end{prop}

The map $\Delta : \Ui_q(\mathfrak{sl}_{2r+2}) \mapsto \Ui_q(\mathfrak{sl}_{2r+2}) \otimes \U_q(\mathfrak{sl}_{2r+2})$ is coassociative, i.e., we have
$(1 \otimes \Delta) \Delta = (\Delta \otimes 1)\Delta: \Ui_q(\mathfrak{sl}_{2r+2}) \longrightarrow \Ui_q(\mathfrak{sl}_{2r+2}) \otimes \U_q(\mathfrak{sl}_{2r+2}) \otimes \U_q(\mathfrak{sl}_{2r+2}).
$
This $\Delta$ 
will be called the { coproduct} of $\Ui_q(\mathfrak{sl}_{2r+2})$, and 
$ \epsilon : \Ui_q(\mathfrak{sl}_{2r+2}) \rightarrow \Qq$ will be called the {counit} of $\Ui_q(\mathfrak{sl}_{2r+2})$.  
The counit map $\epsilon$ makes $\Qq$ a (trivial) $\Ui_q(\mathfrak{sl}_{2r+2})$-module. 
Let $m : \U_q(\mathfrak{sl}_{2r+2}) \otimes \U_q(\mathfrak{sl}_{2r+2}) \rightarrow \U_q(\mathfrak{sl}_{2r+2})$ denote the multiplication map. 
We have 
$
m (\epsilon \otimes 1)\Delta = \imath : \Ui_q(\mathfrak{sl}_{2r+2}) \longrightarrow \U_q(\mathfrak{sl}_{2r+2})
$ by direct computation.


\subsection{The $\imath$-canonical bases} 
In the rest of the section we shall develop the theory of $\imath$-canonical bases for the quantum symmetric pairs  $(\U_q(\mathfrak{sl}_{2r+1}), \Ui_q(\mathfrak{sl}_{2r+1}))$,  \linebreak $(\U_q(\mathfrak{sl}_{2r+2}), \Ui_q(\mathfrak{sl}_{2r+2}))$ and its applications.  
The formulation of the theory is uniform for both cases. Hence we shall drop the subscript, and denote both quantum symmetric pairs simply by $(\U, \Ui)$, and denote the correspond index sets simply by $\I = \I_k$ and $\I^{\imath} = \I^{\imath}_k$, for $k= 2r+2$ or $2r+1$. In this section we shall assume all modules are finite dimensional.

Let $\widehat{\U}$ be the completion of the $\Qq$-vector space $\U$ 
with respect to the following descending sequence of subspaces 
$\U^+ \U^0 \big(\sum_{\hgt(\mu) \geq N}\U_{-\mu}^- \big)$,   for $N \ge 1.$
Then we have the obvious embedding of $\U$ into $\widehat{\U}$. 
We let $\widehat{\U}^-$ be the closure of $\U^-$ in $\widehat{\U}$, and so $\widehat{\U}^- \subseteq \widehat{\U}$. 
By continuity the $\Q(q)$-algebra structure 
on $\U$ extends to a $\Q(q)$-algebra structure on $ \widehat{\U}$.  The bar involution \,$\bar{\ }$\, on $\U$ extends 
by continuity to an anti-linear involution on $\widehat{\U}$, also denoted by \,$\bar{\, }$. The following proposition is the counterpart of \cite[\S2.3, \S2.4 and \S4.4]{BW13}.

\begin{prop}
There is a unique family of elements $\Upsilon_\mu \in {}_\mA\U_{-\mu}^-$ for $\mu \in {\N}{\Pi}$ such that $\Upsilon_0 = 1$, and
$\Upsilon = \sum_{\mu}\Upsilon_\mu \in \widehat{\U}^-$
intertwines the bar involution $\psi_{\imath}$ on $\bun$ and the bar involution $\psi$ on $\U$ via the embedding $\imath$;  
that is, $\Upsilon$ satisfies the following identity (in $\widehat{\U}$): 
\begin{equation}\label{eq:star}
 \imath(\psi_{\imath}{u}) \Upsilon = \Upsilon\  \psi({\imath(u)}), \quad \text{ for all } u \in \bun.
\end{equation}
Moreover, $\Upsilon_\mu = 0$ unless $\mu^{\inv} = \mu$. We also have $\Upsilon \cdot \ov{\Upsilon} =1.$
\end{prop}

Consider a $\Qq$-valued function $\zeta$ on $\Lambda$ such that 
\begin{align}
\zeta (\mu+\alpha_0)&=-q \zeta (\mu) \quad (\text{only for the pair $(\U_q(\mathfrak{sl}_{2r+1}), \Ui_q(\mathfrak{sl}_{2r+1})$}), 
 \notag \\ 
\zeta (\mu+\alpha_i) &= -q^{(\alpha_i, \mu+\alpha_i)-(\alpha_{-i},\mu)} \zeta (\mu), 
  \label{eq:zeta0} \\
\zeta (\mu+\alpha_{-i}) &= -q^{(\alpha_{-i}, \mu+\alpha_{-i}) - (\alpha_{i}, \mu)-1} \zeta (\mu),
\quad \forall \mu \in \Lambda, \; i \in \Ihf. 
\notag
\end{align}
Such $\zeta$ clearly exists (but not unique). For any weight $\U$-module $M$, 
define a $\Qq$-linear map on $ M$ as follows:
\begin{align}
  \label{eq:zeta}
\begin{split}
\widetilde{\zeta}&: M \longrightarrow M, 
 \\
\widetilde{\zeta}  (m &)  =  \zeta (\mu)m, \quad \forall   m \in M_{\mu}.
\end{split}
\end{align}

Recall that $w_0$ is the longest element of $W$ and   $T_{w_0}$  is the associated braid group element
from Section~ \ref{subsec:CB}. The following proposition is the analog of \cite[Theorem~2.18]{BW13}.

\begin{prop}\label{prop:mcT}
For any finite-dimensional $\U$-module $M$, 
the composition map
\[
\mc{T} := \Upsilon\circ \widetilde{\zeta} \circ T_{w_0}: M \longrightarrow M
\] 
is a $\bun$-module isomorphism.
\end{prop}

Recall the bar involutions on $\U$ and its modules are denoted by $\Abar$, and
the bar involution on $\bun$ is denoted by $\Bbar$. 
It is also understood that $\Abar(u) =\Abar(\imath(u))$ for $u\in \bun$.
 
We call
a $\bun$-module $M$ equipped with an anti-linear involution $\Bbar$ {\em involutive} (or {\em $\imath$-involutive}) if
$
\Bbar(u m) = \Bbar(u) \Bbar(m)$, $\forall u \in \bun, m \in M$. For any involutive $\U$-module $M$ with anti-linear involution $\Abar$, the anti-linear involution 
\[
\Bbar := \Upsilon \circ \Abar : M \longrightarrow M 
\]
 makes $M$ an $\imath$-involution $\Ui$-module (cf. \cite[Proposition~3.10]{BW13}). In particular, since we know both $L(\la)$ and ${^{\omega}L}(\lambda)$ are involutive $\U$-modules, they are $\imath$-involutive $\Ui$-modules as well. The following theorem is the counterpart of \cite[Proposition~4.20]{BW13}.


%

%
%
%
%
%
%
%

\begin{thm}\label{thm:BCB}
Let $\la \in \La^+$. The $\bun$-module  ${^{\omega}L}(\lambda)$ admits a unique basis 
\[
\B^\imath(\lambda) := \{T^{\lambda}_{b} \mid b \in \B(\lambda)\}
\]
which is $\Bbar$-invariant and of the form
\[
T^{\lambda}_{b} = b^+ \xi_{-\lambda} +\sum_{b' \prec b}
t^{\lambda}_{b;b'} b'^+  \xi_{-\lambda},
\quad \text{ for }\;  t^{\lambda}_{b;b'} \in q\Z[q].
\]
\end{thm}


\begin{definition}  \label{def:CB}
$\B^\imath(\lambda)$ is called the $\imath$-canonical basis of
the $\bun$-module ${^{\omega}L(\lambda)}$.
\end{definition}

Recall in \cite[Chapter 27]{Lu94} Lusztig has developed a theory of based $\U$-modules $(M,B)$ (for a general quantum group $\U$ of finite type).
The basis $B$ generates a $\Z[q]$-submodule $\mc{M}$ and an $\mA$-submodule ${}_\mA M$ of $M$.
 
\begin{thm} 
Let $(M,B)$ be a finite-dimensional based $\U$-module.  
 
\begin{enumerate}
\item
The $\bun$-module $M$ admits a unique basis (called $\imath$-canonical basis)
$
B^\imath  := \{T_{b} \mid b \in B \}
$
which is $\Bbar$-invariant and of the form
\begin{equation} \label{iCB}
T_{b} = b +\sum_{b' \in B, b' \prec b}
t_{b;b'} b',
\quad \text{ for }\;  t_{b;b'} \in q\Z[q].
\end{equation}

\item
$B^\imath$ forms an $\mA$-basis for the $\mA$-lattice ${}_\mA M$, and
$B^\imath$ forms a $\Z[q]$-basis for the $\Z[q]$-lattice $\mc{M}$.
\end{enumerate}
\end{thm}   

Recall that a tensor product of finite-dimensional simple $\U$-modules is a based $\U$-module
by \cite[Theorem~27.3.2]{Lu94}.

\begin{cor} \label{cor:iCBontensor}
Let $\la_1, \ldots, \la_r \in \La^+$. The tensor product of finite-dimensional simple $\U$-modules
${^{\omega} L (\la_1)}  
\otimes \ldots \otimes {^{\omega} L (\la_r)}$
admits a unique $\Bbar$-invariant basis of the form \eqref{iCB} (called $\imath$-canonical basis).
\end{cor}

 \begin{rem}\label{rem:BW16}
 The construction of the $\imath$-canonical bases in this paper follows straightforwardly from \cite{BW13}. In the ongoing work \cite{BW16}, we construct the $\imath$-canonical bases for general quantum symmetric pairs. In their preprint \cite{BK15}, Balagovic and Kolb constructed the intertwiners for general quantum symmetric pairs (with some overlap with our \cite{BW16}), which leads to the universal solutions of the (quantum) reflection equation (a generalization of the Yang-Baxter equation). 

\end{rem}


\section{Dualities}
In this section, we study various dualities between the coideal algebras and the Hecke algebras of type B/C/D. The theory is again uniform in most cases, except in subsection~\ref{subsec:C}, where we only study the duality between $\Ui_q(\mathfrak{sl}_{2r+1})$ and $\mc{H}_{C_n}$. Hence we shall simplify the notation (except in subsection~\ref{subsec:C}), and denote both quantum symmetric pairs simply by $(\U, \Ui)$, and denote the correspond index sets simply by $\I = \I_k$ and $\I^{\imath} = \I^{\imath}_k$, for $k= 2r+2$ or $2r+1$.

\subsection{The $(\Ui, \mc{H}^1_{B_m})$-duality}
We set 
$I = I_k = \{a \pm \hf \vert a \in \I = \I_k\}$. 
Let the $\Qq$-vector space $\VV := \sum _{a \in I}\Qq v_{a} $ be the natural representation of $\U$, 
hence a $\Ui$-module. The action of $\U$ on $\VV$ is given by ($i \in \I$ and $a \in I$) 
\[E_{\alpha_i} v_{a} = \delta_{i+\hf, a} v_{a-1}, \quad F_{\alpha_i} v_{a} = \delta_{i-\hf,a} = v_{a+1} \quad \text{ and } \quad K_{\alpha_i} v_{a} = q^{\delta_{i-\hf,a} - \delta_{i+\hf},a} v_a. 
\]
We shall call $\VV$ the natural representation of $\Ui$ as well. 
For $m \in {\Z_{> 0}}$, 
$\VV^{\otimes m}$ becomes a natural $\U$-module (hence a $\Ui$-module) 
via the iteration of the coproduct $\Delta$. 
Note that $\VV$ is an involutive  $\U$-module with $\Abar$ defined as 
\[\Abar(v_a) :=v_a, \quad \text{ for all } a \in I.
\]
Therefore $\VV^{\otimes m}$ is an involutive  $\U$-module 
and hence an $\imath$-involutive $\Ui$-module.

Let $W_{B_m}$ be the Coxeter groups of type $\text{B}_m$ with simple reflections $s_j, 0 \leq j \leq m-1$,
where the subgroup  generated by $s_i$, $1\leq i \leq m-1$ is isomorphic to $W_{A_{m-1}} \cong \mathfrak{S}_m$. 
The group  $W_{B_m}$ and its subgroup $\mathfrak{S}_m$
act naturally on $I^m$ on the right as follows: for any $f \in I^m$, $ 1 \leq i \leq m$, we have 
\begin{equation}  \label{eq:rightW}
f \cdot s_j = 
 \begin{cases}
 (\dots, f(j+1), f(j), \dots) , &\text{if }  j > 0;\\
 (-f(1), f(2), \dots, f(m)),& \text{if } j =0.
 \end{cases}
 \end{equation}
 
 Let $\mathcal{H}^p_{B_m}$ be the Iwahori-Hecke algebra of type $B_m$ with two parameters $p$ and $q$ over $\mathbb Q(q, p)$. 
It is generated by $H^p_0, H_1, H_2, \dots , H_{m-1}$, subject to the following relations ($i, j >0$),
\begin{equation}\label{eq:HeckeB}
\begin{split}
(H^p_0-p^{-1})(H^p_0 +p) &= 0 \quad \text{ and } \quad (H_i -q^{-1})(H_i +q) = 0,  
 \\
H_i H_{i+1} H_i &= H_{i+1} H_i H_{i+1}, 
 \\
H_i H_j &= H_j H_i, \qquad \qquad\qquad\qquad  \quad\qquad \qquad \text{for } |i-j| >1, 
\\
H^p_0 H_1 H^p_0 H_1&=H_1H^p_0 H_1H^p_0 \quad \text{ and } \quad H^p_0 H_i = H_i H^p_0 , \qquad \text{for } i >1.
\end{split}
\end{equation}
The bar involution on $\mathcal
H^p_{B_m}$ is the unique anti-linear ($\overline{q} =q^{-1}$ and $\overline{p} = p^{-1}$) automorphism defined by
$\overline{H_i} =H_{i}^{-1}$ and $\overline {H^p_0} = (H^p_0)^{-1}$.
 
Let $\mathcal{H}^1_{B_m}$ be the degenerate Iwahori-Hecke algebra of type $B_m$ over $\mathbb Q(q)$ with the parameter $p=1$. We shall write the generator $H^p_0$ as $s_0$ in this case. Note that we have $s_0^2 =1$ and $\overline{s_0} = s_0$.

For any $f \in I^m$, we can view $f$ as a function from the set $\{1, 2, \dots, m\}$ to $I^m$. Thus we define
$
M_f= v_{f(1)} \otimes \cdots \otimes v_{f(m)}.
$
The Weyl group $W_{B_m}$ acts on $I^m$ by \eqref{eq:rightW} as before. 
Now the degenerate Hecke algebra $\HBm^1$  acts on the $\Qq$-vector space $\VV^{\otimes m}$ as follows ($ a > 0$): 
\begin{align*} 
 M_f \cdot H_i &=
 \begin{cases}
 q^{-1}M_f, & \text{ if }  f(i) = f(i+1);\\
 M_{f \cdot  s_i}, & \text{ if }  f(i) < f(i+1);\\
 M_{f \cdot  s_i} + (q^{-1} - q) M_{f}, & \text{ if } f(i) > f(i+1);
 \end{cases}\\
 M_f \cdot s_0 &= M_{f \cdot s_0}.
\end{align*}

Introduce the $\Qq$-subspaces of $\VV$:
\begin{align*}
\VV_{-}  &=\bigoplus_{i \in I} \Qq (v_{-i} - v_{i}),
  \\
\VV_{+}  &= \bigoplus_{i \in I} \Qq (v_{-i} +  v_{i} ).
\end{align*}

The following lemma follows from direct computation.

\begin{lem}  \label{int:lem:V+-}
$\VV_-$ and $\VV_+$ are $\Ui$-submodules of $\VV$.
Moreover, we have   
$
\VV = \VV_- \oplus \VV_+.
$
\end{lem}

 Let $s$ be the largest number in $\I$. Now we fix 
$\zeta$ in \eqref{eq:zeta}  such that $\zeta (\varepsilon_{-s}) = 1$. It follows that
$$
\zeta({\varepsilon_{s -i}}) =
 (-q)^{-2s+i}, \qquad \text{ for } s-i \in \I.
 $$
Let us compute the $\Ui$-homomorphism
$\mc{T} = \Upsilon\circ \widetilde{\zeta} \circ T_{w_0}$ (see Proposition~\ref{prop:mcT}) 
on the $\U$-module $\VV$; we recall that $w_0$ here is associated to $\U$ instead of $W_{B_m}$ or $W_{A_{m-1}}$. 

\begin{lem}  \label{int:lem:T}
The $\Ui$-isomorphism $\mc{T}^{-1}$ on $\VV$ acts as $(- \id)$ on the submodule $\VV_-$ and 
as $ \id$ on the submodule $\VV_+$. 
\end{lem}
 
\begin{proof}
First one computes that the action of $T_{w_0}$ on $\VV$ is given by 
$$
T_{w_0} (v_{-s+i}) =(-q)^{2s-i}v_{s-i}, 
\qquad \text{ for }  0 \le i \le 2s.
$$ 
Hence 
\begin{equation}\label{eq:mcTtos0C}
\widetilde{\zeta} \circ T_{w_0}(v_a) =
v_{ a\cdot s_0}
\end{equation}

One computes the first few terms of $\Upsilon$. For example we have $\Upsilon_{\alpha_{-\hf}+\alpha_{\hf}} = -\qq F_{\alpha_{-\hf}}F_{\alpha_{\hf}}$ for the quantum symmetric pair $(\U_q(\mathfrak{sl}_{2r+2}), \Ui_q(\mathfrak{sl}_{2r+2}))$.
Therefore using $\mc{T} = \Upsilon\circ \widetilde{\zeta} \circ T_{w_0}$ we have (for $i \in \I^{\imath}$)
\begin{align}
\mc{T}^{-1} v_0 &=  v_0\label{int:eq:mcT1} \quad (\text{for the pair $(\U_q(\mathfrak{sl}_{2r+2}), \Ui_q(\mathfrak{sl}_{2r+2}))$}),\\
\mc{T}^{-1} ( v_{-i} - v_{i}) &= (-1) (v_{-i} - v_{i})\label{int:eq:mcT2},\\
\mc{T}^{-1} (v_{-i}+v_{i}) &= (v_{-i}+v_{i})\label{int:eq:mcT3}.
\end{align}
The lemma now follows from Lemma~\ref{int:lem:T},
since  $\mc{T}^{-1}$ is a $\Ui$-isomorphism.
\end{proof}

\begin{rem}
We remind the readers that the intertwiner $\Upsilon$ associated with the quantum symmetric pair $(\U, \Ui)$ in this paper is different from the one in \cite{BW13}, since we are considering quantum symmetric pairs with different parameters. This leads to different actions of $\mc{T}$ on the natural representation $\VV$ of $\U$ (c.f. \cite[Lemma~5.3]{BW13}). 
\end{rem}

We have the following generalization of Schur-Jimbo duality, whose proof follows from \cite[Theorem~5.4]{BW13}.

\begin{thm}[$(\Ui, \mc{H}^1_{B_m})$-duality] \label{int:thm:SchurB}
\begin{enumerate}
\item
The action of $\mc{T}^{-1} \otimes \id^{m-1}$ coincides with the action of 
$s_0 \in \mathcal{H}^1_{B_m}$ on $\VV^{\otimes m}$. 

\item
The actions of $\Ui$ and $\mathcal{H}^1_{B_m}$ on $\VV^{\otimes m}$ 
commute with each other, and they form double centralizers. 
\end{enumerate}
\end{thm} 

\begin{rem}
By variations of the choices of the subalgebras  (Proposition~\ref{int:prop:embedding} and Proposition~\ref{prop:embedding}), we can obtain the coideal subalgebras $\Ui$ that forms double centralizers with the Hecke algebra $\mathcal{H}^p_{B_m}$ of two parameters, when acting on the tensor product $\VV^{\otimes m}$.
\end{rem}




\subsection{The Hecke algebra $\mathcal{H}_{D_m}$ of type D}\label{subsec:UiD}

Let $W_{D_m}$ be the Coxeter group of type $D_m (m \ge 2)$ with simple reflections $s^d_0$ and $s_j$, $1 \le j \le m-1$. The Coxeter group $W_{D_m}$ can be realized as a subgroup of $W_{B_m}$ via the following embedding : $s^d_0 \mapsto s_0 s_1 s_0$ and $s_i \mapsto s_i$ for $i \ge 1$. When $m=1$, we understand $W_{D_m}$ as the trivial group. The Weyl group $W_{D_m}$ acts on the set $I^m$ via the embedding. 

Let $\mathcal{H}_{D_m}$ be the Iwahori-Hecke algebra of type $D_m$ over $\Qq$. It is generated by $H_0$, $H_1$, $\dots$, $H_{m-1}$, subject to the following relations:
\begin{align*}
(H_i -q^{-1})(H_i +q) &= 0,   & \text{for } i \geq 0, &
 \\
H_i H_{i+1} H_i &= H_{i+1} H_i H_{i+1}, & \text{for } i> 0,&
 \\
H_i H_j &= H_j H_i, & \text{for } |i-j| >1, &
\\
H_0 H_2 H_0 = H_2 H_0 H_2 \quad &\text{ and } \quad H_0H_i=H_iH_0,  &\text{for } i \neq 2.
\end{align*}
The bar involution on $\mathcal
H_{D_m}$ is the unique anti-linear involution  defined by
$\overline{H_i} =H_{i}^{-1}$ and $\overline{q} =q^{-1},$
for all $ 0 \le i \le m-1$.

\begin{lem}\label{lem:HDtoHB}
There is a $\Qq$-algebra embedding $\rho : \mc {H}_{D_m} \rightarrow \mc{H}^1_{B_m}$ such that 
\[
\rho (H_0) = s_0 H_1 s_0 \quad \text{ and } \quad \rho (H_i) = H_i, \quad \text{for } i \ge 1.
\]
 Moreover, $\rho$ commutes with the bar involutions, that is, $\rho (\overline {h}) = \overline {\rho (h)}$ for $h \in \mc{H}_{D_m}$. (The bar involution on the left hand side is the bar involution on $\mc{H}_{D_m}$, while the bar involution on the right hand side is the bar involution on $\mc{H}^1_{B_m}$.)
\end{lem}

\begin{proof}
It suffices to check the relations involving $H_0$. Note that we have 
\begin{align*}
 (s_0 H_1 s_0 - q^{-1})( s_0 H_1 s_0 +q) &= s_0(H_1 - q^{-1})(H_1 +q ) s_0 = 0, \\
  s_0 H_1 s_0 \cdot H_2  \cdot s_0 H_1 s_0  =  s_0 H_1 H_2 H_1 s_0 &=  s_0 H_2 H_1 H_2 s_0 = H_2 \cdot s_0 H_1 s_0 \cdot H_2,\\
  s_0 H_1 s_0 \cdot H_i &= H_i \cdot s_0 H_1 s_0 \quad \text{ for } i \neq 2, 0.
\end{align*}
This shows that $\rho$ is a homomorphism of $\Qq$-algebras. Then for dimension reason, we see that $\rho$ is an embedding. 

To show that $\rho$ commutes with the bar involutions, it suffices to show that $\rho(\overline{H_0}) = \overline{\rho (H_0)}$. This is clear since (recall $\overline{s_0} = s_0$)
\[
\rho (\overline{H_0}) = \rho (H_0^{-1}) = (s_0 H_1 s_0 ) ^{-1} = s_0 H^{-1}_1 s_0 = \overline{\rho (H_0)}.
\]
The lemma follows.
\end{proof}

Via the embedding $\rho :  \mc {H}_{D_m} \rightarrow \mc{H}^1_{B_m}$, the Hecke algebra $\mc{H}_{D_m}$ has a natural action on the tensor space $\VV^{\otimes m}$ as follows  (Note that $(f(1), f(2), \dots ) \cdot s^d_0  = (-f(2), -f(1), \dots)$): 
\begin{equation}  \label{int:eq:HBm}
 M_f H_a=
 \begin{cases}
 q^{-1}M_f, & \text{ if } a>0, f(a) = f(a+1);\\
 M_{f \cdot  s_a}, & \text{ if } a > 0, f(a) < f(a+1);\\
 M_{f \cdot  s_a} + (q^{-1} - q) M_{f}, & \text{ if } a > 0, f(a) > f(a+1);\\
 M_{f \cdot  s^d_0}, & \text{ if } a = 0, -f(1) < f(2) ;\\
 M_{f \cdot  s^d_0} + (q^{-1} -q)M_f, & \text{ if } a=0, -f(1) > f(2);\\
 q^{-1} M_f, & \text{ if } a =0, -f(1) = f(2).
 \end{cases}
\end{equation}

The following corollary follows immediately from Theorem~\ref{int:thm:SchurB}.
\begin{cor}
\begin{enumerate}
	\item	The action of $((\mc{T}^{-1} \otimes \id)  \cdot \mc{R}^{-1} \cdot (\mc{T}^{-1} \otimes \id)) \otimes \id^{m-2}$ coincides with the action of $H_0 \in \mathcal{H}_{D_m}$ on $\VV^{\otimes m}$. 
	\item	The actions of $\Ui$ and $\mathcal{H}_{D_m}$ on $\VV^{\otimes m}$ 
commute with each other.
\end{enumerate}
\end{cor}

\begin{rem}
The commuting relation of the actions of $\Ui$ and $\mathcal{H}_{D_m}$ on $\VV^{\otimes m}$  has also been observed in  \cite[\S7.6]{ES13} by direct computation without using the Theorem~\ref{int:thm:SchurB}.
\end{rem}

\begin{definition}
An element $f \in I^m$ is called ($D$-)anti-dominant, if $  |f(1) | \le f(2) \le f(3) \cdots \le f(m)$.  ($ |f(1) |$ denotes the absolute value of $f(1)$.)
\end{definition}

\begin{thm}\label{thm:samebar}
The bar involution $\ibar: \VV^{\otimes m} \rightarrow \VV^{\otimes m}$ is compatible with both the bar involution of $\mathcal{H}_{D_m}$ and the bar involution of $\Ui$; 
that is, for all $v \in \VV^{\otimes m}$,  $h \in \mathcal{H}_{D_m}$, and $u \in \Ui$, we have 
\[
\ibar(u v h) = \ibar(u) \, \ibar(v) \ov{h}  \quad \text{ and } \quad \ibar(M_f) = M_f \text{ for all $D$-anti-dominant } f .
\]
Moreover such bar involution on $\VV^{\otimes m}$ is unique.
\end{thm}

\begin{proof}
The exact same proof as \cite[Theorem~5.8]{BW13} shows that the bar involution $\ibar = \Upsilon \circ \Abar$ is compatible with both the bar involution of $\mathcal{H}^1_{B_m}$ and the bar involution of $\Ui$. But since the embedding $\rho: \mc{H}_{D_m} \rightarrow \mc{H}^1_{B_m}$ is compatible with bar involutions (Lemma~\ref{lem:HDtoHB}), we know that $\ibar$ is compatible with the bar involution of $\mc{H}_{D_m}$. Therefore we only need to show that $\ibar(M_f) = M_f \text{ for all $D$-anti-dominant } f $.

For any $D$-anti-dominant $f \in I^m$ with $0 \le f(1)$, we have $\ibar(M_f) = M_f$ by \cite[Theorem~5.8]{BW13}. For any $D$-anti-dominant $f \in I^m$ with $f(1) < 0$, we see that $f \cdot s_0$ is still $D$-anti-dominant and $0 \le -f(1) = f(1) \cdot s_0$. We have 
\[
\ibar(M_f) =  \ibar (M_{f \cdot s_0} s_0) = \ibar (M_{f \cdot s_0})  \overline{s_0} = M_{f \cdot s_0} s_0 = M_f.
\]
Thus $\ibar (M_f) = M_f$ for all $D$-anti-dominant $f \in I^m$.

The uniqueness of such bar involution on $\VV^{\otimes m}$ follows from a  standard argument (cf. \cite[Theorem~5.8]{BW13}). The theorem follows. 
\end{proof}

It is well-known that via the action defined in \eqref{int:eq:HBm}, the tensor product $\VV^{\otimes m}$ becomes a direct sum of permutation modules of $\mc{H}_{D_m}$. Therefore the (parabolic) Kazhdan-Lusztig basis of $\mc{H}_{D_m}$ induces a (parabolic) Kazhdan-Lusztig basis (of type D) on $\VV^{\otimes m}$. Recall that $\VV^{\otimes m}$ admits an $\imath$-canonical basis by Corollary~\ref{cor:iCBontensor}. The following Corollary follows immediately from Theorem~\ref{thm:samebar}.

\begin{cor}\label{cor:samebar}
The $\imath$-canonical basis on the tensor space $\VV^{\otimes m}$ is the same as the Kazhdan-Lusztig basis of type $D$.
\end{cor}

\begin{rem}
Theorem~\ref{thm:samebar} and Corollary~\ref{cor:samebar} make sense in the case $m=1$ as well, where we understand $W_{D_m}$ as the trivial group. More precisely, the $\imath$-canonical basis on $\VV$ is the same as the canonical basis on $\VV$.
\end{rem}

\subsection{Bruhat orderings}\label{subsec:Bruhat}
In this subsection we show that the bar involution $\ibar$ on $\VV^{\otimes m}$ satisfies the type D Burhat ordering, which should be expected in light of Theorem~\ref{thm:samebar}. We do not need results from this section for any other part of this paper. 

Let $X(m) = \hf \Z [\epsilon_1, \epsilon_2, \dots \epsilon_m]$ and set
$$
\rho = (0 \epsilon_1) - \epsilon_2  - \cdots - (n-1)\epsilon_n.
$$ There is a non-degenerate symmetric bilinear form $(\cdot \vert \cdot)$ on $X(m)$ such that 
$(\epsilon_i \vert \epsilon_j) = \delta_{ij}
$.

There is a natural injective map $I^m \rightarrow X(m)$, defined as 
\begin{align*}
f &\mapsto \lambda_f , \text{ where } \lambda_f = \sum^m_{i=1} f(i) \epsilon_i - \rho, \quad \text{ for } f \in I^m.
\end{align*}

For any $f \in I^m$, the $\U$-weight of $M_f$ is ${\rm wt}(f) = \sum^m_{i=1} \varepsilon_{f(i)} \in \Lambda$. We define the $\Ui$-weight of $M_f$ to be ${\rm{wt}_{\imath}}(f) = \sum^m_{i=1} \overline{\varepsilon_i}$, i.e., the image of ${\rm wt}(f)$ in the quotient $\Lambda_{\inv}$. 

Note that we always have $\lambda_f - \lambda_g \in \Z[\epsilon_1, \dots, \epsilon_m]$ for any $f, g \in I^{m}=I^{m}_k$ (for both $k = 2r+1 \text{ or } 2r+2$).

\begin{definition} We define the following two partial orderings on $I^m$.

\begin{enumerate}
\item 	For any $f, g \in I^m$, we say $g \preceq_B f$ if 
\[
\rm{wt}_{\imath} (f) = \rm{wt}_{\imath} (g) \quad \text{ and }\quad  \lambda_f - \lambda _g  = a_0 (-\epsilon_1) + \sum^{m-1}_{i=1} a_i (\epsilon_i - \epsilon_{i+1}), \text{ where }a_i \in \N.\]
\item	For any $f, g \in I^m$, we say $g \preceq_D f$ if 
\[
g \preceq_B f \quad \text{ and } \quad g\cdot s_0  \preceq_B f \cdot s_0. 
\](Recall $ (f(1), \dots) \cdot s_0 = (-f(1), \dots)$.)
\end{enumerate}

\end{definition}

\begin{prop}\label{prop:bruhatD} 
Let $g ,f \in I^m$ such that $g \preceq_D f$. 
\begin{enumerate} 
\item	If $m=1$, then $f = g$.
\item
If $m \ge 2$, then we have 
\[
\lambda_f - \lambda_g = a_0 (-\epsilon_1 - \epsilon_2 ) + \sum^{m-1}_{i=1} a_i (\epsilon_i - \epsilon_{i+1}), \qquad \text{ where } a_i \in \N.
\]
\end{enumerate}
\end{prop}

\begin{proof}
When $m =1$, the proposition follows from direct computation. So let us assume $m \ge 2$. It suffices to consider the case where 
\[
\lambda_f - \lambda_g \in \Z[\epsilon_1, \epsilon_2].
\]
Otherwise we can always find $h \in I^m$ (and then replace $g$ by $h$) such that $g \preceq_D h \preceq_D f$ and 
$
\lambda_h - \lambda_g \in \Z[\epsilon_3, \epsilon_4, \dots, \epsilon_m]$, $\lambda_f -\lambda_h \in \Z[\epsilon_1, \epsilon_2]$. Thus let us simply assume $m=2$.

We know that ${\rm wt}_{\imath}(f) = {\rm wt}_{\imath}(g)$ by our assumption. All elements in $I^{2}$ of the same  $\Ui$-weight ${\rm wt}_{\imath}(f)$ has the Hasse diagram with respect to the partial ordering $\preceq_B$ as the following ($a \ge 0$, $b \ge 0$, $a \le b$):

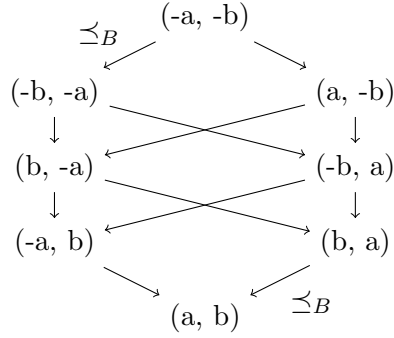
\begin{center}
\begin{tikzpicture}

\draw (0,0) node (0101) {(-a, -b)};
\draw (-2,-1) node (010) {(-b, -a)};
\draw (-2, -2) node (10) {(b, -a)};
\draw (-2, -3) node (0) {(-a, b)};
\draw (0, -4) node (e) {(a, b)};
\draw (2, -1) node (101) {(a, -b)};
\draw (2, -2) node (01) {(-b, a)};
\draw (2, -3) node (1) {(b, a)};


\draw[->] (1) -- node[below right] {$\preceq_B$} (e);
\draw[->] (01) -- (1);
\draw[->] (101) -- (01);
\draw[->] (0101) -- (101);
\draw[->] (0) --  (e);
\draw[->] (10) -- (0);
\draw[->] (010) -- (10);
\draw[->] (0101) --  node[above left] {$\preceq_B$} (010);
\draw[->] (010) -- (01);
\draw[->] (101) -- (10);
\draw[->] (10) -- (1);
\draw[->] (01) -- (0);

\end{tikzpicture}
\end{center}
Applying $s_0$ to the vertices, which preserves the $\Ui$-weight ${\rm wt}_{\imath}(f)$, we can rewrite the Hasse diagram with respect to $\preceq_B$ as: 
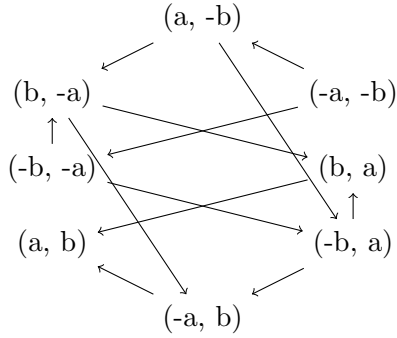
\begin{center}
\begin{tikzpicture}

\draw (0,0) node (101) {(a, -b)}  ;
\draw (-2,-1) node  (10) {(b, -a)} ;
\draw (-2, -2) node (010) {(-b, -a)};
\draw (-2, -3) node (e) {(a, b)}  ;
\draw (0, -4) node (0) {(-a, b)};
\draw (2, -1) node (0101) {(-a, -b)} ;
\draw (2, -2) node (1) {(b, a)} ;
\draw (2, -3) node (01) {(-b, a)} ;

\draw[->] (1) --  (e);
\draw[->] (01) -- (1);
\draw[->] (101) -- (01);
\draw[->] (0101) -- (101);
\draw[->] (0) -- (e);
\draw[->] (10) -- (0);
\draw[->] (010) -- (10);
\draw[->] (0101) -- (010);
\draw[->] (010) -- (01);
\draw[->] (101) -- (10);
\draw[->] (10) -- (1);
\draw[->] (01) -- (0);

\end{tikzpicture}
\end{center}

Combining the two diagrams, we have the following Hasse diagram with respect to the partial ordering $\preceq_D$: 

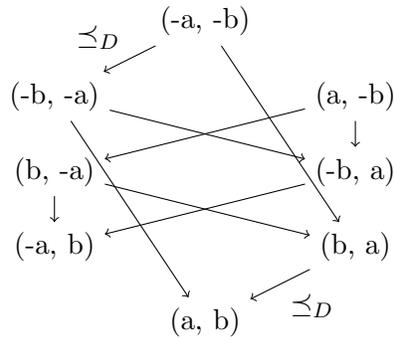
\begin{center}
\begin{tikzpicture}

\draw (0,0) node (0101) {(-a, -b)};
\draw (-2,-1) node (010) {(-b, -a)};
\draw (-2, -2) node (10) {(b, -a)};
\draw (-2, -3) node (0) {(-a, b)};
\draw (0, -4) node (e) {(a, b)};
\draw (2, -1) node (101) {(a, -b)};
\draw (2, -2) node (01) {(-b, a)};
\draw (2, -3) node (1) {(b, a)};


\draw[->] (0101) --  node[above left] {$\preceq_D$} (010);
\draw[->] (0101) -- (1);
\draw[->] (010) -- (e);
\draw[->] (010) -- (01);
\draw[->] (101) -- (10);
\draw[->] (10) -- (0);
\draw[->] (10) -- (1);
\draw[->] (01) -- (0);
\draw[->] (101) -- (01);
\draw[->] (1) -- node[below right] {$\preceq_D$}  (e);

\end{tikzpicture}
\end{center}
The rest of the proposition follows from case by case computation. For example, we have 
\[
\lambda_{(-b,-a)} -\lambda_{(-b, a)} = -2a \epsilon_2 = a (-\epsilon_1 - \epsilon_2) + a (\epsilon_1- \epsilon_2).
\]
\end{proof}


\begin{rem}
Note that the set $\{\epsilon_1, \epsilon_1- \epsilon_2, \dots,  \epsilon_{m-1}- \epsilon_m \}$, and the set $\{\epsilon_1 + \epsilon_2, \epsilon_1- \epsilon_2, \dots,  \epsilon_{m-1}- \epsilon_m \}$ are the sets of simple roots for type B, and type D root systems, respectively. So for $f$, $g \in I^m$, $g \preceq_B f$ means that $\lambda_f -\lambda_g $ is a non-negative integral linear combination of type B simple roots, and $g \preceq_D f$ means that $\lambda_f -\lambda_g $ is a non-negative integral linear combination of type D simple roots, respectively. 
\end{rem}

\begin{rem}
Actually if we know $g \preceq_B f$ and $g \cdot s_0 \preceq_B f \cdot s_0$, we have 
\begin{align*}
\lambda_f - \lambda_g &= a_0 (-\epsilon_1) + \sum^{m-1}_{i=1} a_i (\epsilon_i - \epsilon_{i+1}) \quad & \text{ where } a_i \in \N;\\
\lambda_{f\cdot s_0}- \lambda_{g \cdot s_0} &= (2a_1 -a_0) (-\epsilon_1) + \sum^{m-1}_{i=1} a_i (\epsilon_i - \epsilon_{i+1})  \quad & \text{ where } a_i \in \N, 2a_1 - a_0 \in \N.
\end{align*}
But we can write $\lambda_f - \lambda_g$ as 
\[
\lambda_f - \lambda_g = \hf(a_0) (-\epsilon_{1} - \epsilon_2) + \hf(2a_1 - a_0) (\epsilon_1- \epsilon_2) + \sum^{m-1}_{i=2} a_i (\epsilon_i - \epsilon_{i+1}).
\]
We already know that $a_0 \ge 0$ and $\hf(2a_1 - a_0) \ge 0 $. So Proposition~\ref{prop:bruhatD} is essentially showing that $\hf(a_0)$ and $\hf(2a_1 - a_0)$ are  actually integers, if we in addition have ${\rm wt}_{\imath} (f) = {\rm wt}_{\imath} (g)$.
\end{rem}

In light of the proposition we shall see that the bar involution $\ibar$ on the tensor space $\VV^{\otimes m}$ actually respects the coarser partial  ordering $\preceq_D$.
\begin{prop} 

For any $f \in I^m$, we have 
\[
\ibar (M_f) = \Upsilon \Abar (M_f) = M_f + \sum_{g \preceq_D f} c_{g, f} M_g.
\]
\end{prop}

\begin{proof}

Following \cite[Lemma~9.4]{BW13}, we have 
\[
\ibar (M_f) = \Upsilon \Abar (M_f) = M_f + \sum_{g \preceq_B f} c_{g, f} M_g.
\]
Thanks to compatibility in Theorem~\ref{thm:samebar}, we have 
\[
M_{f\cdot s_0} + \sum_{g' \preceq_B f\cdot s_0} c_{g', f\cdot s_0} M_{g'} = \ibar(M_f \cdot s_0) = \ibar(M_f) \cdot \overline{s_0} =  M_f \cdot s_0 + \sum_{g \preceq_B f} c_{g,f} M_{g \cdot s_0}.
\]
Therefore we have $c_{g', f \cdot s_0} = c_{g, f}$ if $g' = g\cdot s_0$. Thus we have $g \preceq_B f$ and $g\cdot s_0 \preceq_B f \cdot s_0$.
By Proposition~\ref{prop:bruhatD}, this implies $g \preceq_D f$. The proposition follows.
\end{proof}

\begin{rem}
We shall NOT use the partial ordering $\preceq_D$, or any variation of the this partial ordering in this paper. We shall only use the partial ordering $\preceq_B$ and its variants in this paper.

\end{rem}


\subsection{The $(\Ui_q(\mathfrak{sl}_{2r+1}), \mathcal{H}_{C_n})$-duality}\label{subsec:C}
In this subsection, we shall only consider the quantum symmetric pair $(\U_q(\mathfrak{sl}_{2r+1}), \Ui_q(\mathfrak{sl}_{2r+1}))$. In order to avoid confusion, We shall not use the simplified notations in this section.

Let $\WW :=\VV^*$ be the (restricted) dual module of $\VV$ with basis $\{w_a \mid a \in I_{2r+1}\}$ such that 
$\langle w_a, v_b \rangle =  (-q)^{-a} \delta_{a,b}$. The action of $\U_q(\mathfrak{sl}_{2r+1})$ on $\WW$ is given by the following formulas 
(for $i \in \I_{2r+1}$, $a \in I_{2r+1}$):
\[
E_{\alpha_i} w_a = \delta_{i-\hf, a} w_{a+1}, \quad F_{\alpha_i}w_a 
= \delta_{i+\hf, a}w_{a-1}, \quad K_{\alpha_i} w_a = q^{-(\alpha_i, \varepsilon_a)}w_a.
\]
By restriction through the embedding $\iota$, $\WW$ is naturally a $\Ui_q(\mathfrak{sl}_{2r+1})$-modules. For $n \in \Z_{>0}$, $\WW^{\otimes n}$ is naturally a $\U_q(\mathfrak{sl}_{2r+1})$-module, hence a $\Ui_q(\mathfrak{sl}_{2r+1})$-module, via the iteration of the coproduct $\Delta$. Note that $\WW$ is an involutive $\U_q(\mathfrak{sl}_{2r+1})$-module with $\psi$ defined as
\[
 \psi(w_a) = w_a, \quad \text{ for all }a \in I_{2r+1}. 
\]
Therefore $\WW^{\otimes n}$ is an involutive $\U_q(\mathfrak{sl}_{2r+1})$-module and hence an $\imath$-involutive $\Ui_q(\mathfrak{sl}_{2r+1})$-module.

Let $\mathcal{H}_{C_n} =  \mathcal{H}^{q}_{B_n}$ be the Hecke algebra of type $C$ with equal parameters ($p=q$). For $f \in I_{2r+1}^n$, let $M^*_f = w_{f(1)} \otimes \cdots \otimes w_{f(n)} \in \WW^{\otimes n}$.  The Hecke algebra $\mathcal{H}_{C_n}$ acts on $\WW^{\otimes n}$ as follows:
\begin{equation} \label{eq:typeC}
\begin{split}
 M^*_f H_a=&
 \begin{cases}
 q^{-1}M^*_f, & \text{ if } a>0, f(a) = f(a+1);\\
 M^*_{f \cdot  s_a}, & \text{ if } a > 0, f(a) > f(a+1);\\
 M^*_{f \cdot  s_a} + (q^{-1} - q) M_{f}, & \text{ if } a > 0, f(a) < f(a+1).
 \end{cases}\\
 M^*_f H^q_0= &
  \begin{cases}
 M^*_{f \cdot  s_0}, & \text{ if } f(1) < 0;\\
 M^*_{f \cdot s_0} + (q^{-1} - q) M_{f}, & \text{ if } f(1) > 0;\\
 q^{-1} M^*_{f}, &\text{ if }f(1)=0.
 \end{cases}
 \end{split}
\end{equation}

\begin{definition}
An element $f \in I_{2r+1}^n$ is called ($C$-)anti-dominant, if $  0 \le f(1)  \le f(2) \le f(3) \cdots \le f(n)$.  
\end{definition}

Let us fix a choice of $\zeta$ in \eqref{eq:zeta} such that $\mc{T}^{-1} : \WW \rightarrow \WW$ maps $w_{-s}$ to $w_{s}$, where $s$ is the largest number  in $I_{2r+1}$. The tensor product $\WW^{\otimes n}$ becomes a direct sum of permutation modules of $\mathcal{H}_{C_n}$ via the action defined in \eqref{eq:typeC}. Hence $\WW^{\otimes n}$ admits a Kazhdan-Lusztig basis  (of type C). 
The following theorem is the counterpart of the Theorem~\ref{int:thm:SchurB} and Theorem~\ref{thm:samebar}, Corollary~\ref{cor:samebar}. 


\begin{thm}[$(\Ui_{q}(\mathfrak{sl}_{2r+1}), \mathcal{H}_{C_n})$-duality] \label{thm:KLC}
\begin{enumerate}
	\item	The action of $\mc{T}^{-1} \otimes \id^{n-1}$ coincides with the action of $H_{0} \in \mathcal{H}_{C_n}$ on $\WW^{\otimes n}$.
	\item	The actions of $\Ui_{q}(\mathfrak{sl}_{2r+1})$ and $\mathcal{H}_{C_n}$ on $\WW^{\otimes n}$ commute with each other, and they form double centralizers.
	\item There exists a unique bar involution $\psi_{\imath} : \WW^{\otimes n} \rightarrow \WW^{\otimes n}$ such that for all $w \in \WW^{\otimes n}$, $h \in \mathcal{H}_{C_n}$, $u \in \Ui_q(\mathfrak{sl}_{2r+1})$, and all $C$-anti-dominant $f \in I_{2r+1}$, we have 
\[
\ibar(u v h) = \ibar(u) \, \ibar(v) \ov{h}  \quad \text{ and }  \quad \ibar(M_f) = M_f.
\]

\item	The $\imath$-canonical basis on the tensor space $\WW^{\otimes m}$ is the same as the Kazhdan-Lusztig basis of type $C$.
\end{enumerate}
\end{thm}

\begin{rem}
The actions of $\Ui_q(\mathfrak{sl}_{2r})$ and $\mc{H}_{C_n}$ on $\WW^{n}$, the restricted dual of the natural representation $\VV$ of $\U_q(\mathfrak{sl}_{2r})$, do not commute. 
\end{rem}


\section{Kazhdan-Lusztig theory of super type D}\label{sec:rep}
In this section we shall apply the theory of $\imath$-canonical bases from Section~\ref{sec:QSP} to study the BGG category $\mc{O}$ of the Lie superalgebra $\mathfrak{osp}(2m|2n)$ with respect to various Borel subalgebras. We shall formulate and establish the Kazhdan-Lusztig  theory for the Lie superalgebra $\mathfrak{osp}(2m|2n)$. 

We shall first set up various Fock spaces and establish the $\imath$-canonical bases on suitable completions of those Fock spaces. Then we study various versions of the category $\mc{O}$  of the Lie superalgebra $\mathfrak{osp}(2m|2n)$. Finally we can formulate and establish the Kazhdan-Lusztig theory for the BGG category $\mc{O}$ of the Lie superalgebra $\mathfrak{osp}(2m|2n)$.

%

We emphasize that we shall use the same partial orderings as \cite[Definition~8.3]{BW13}, even though \S\ref{subsec:Bruhat} suggests that we can use some coarser partial orderings (type B vs type D). Most proofs shall be similar to \cite{BW13}, hence shall be omitted and referred to \cite{BW13}.


\subsection{Infinite-rank constructions and notations}
We set 
\begin{align}
 \label{eq:III}
\I_{odd} = \cup^{\infty}_{r=0}  & \I_{2r+1} = \Z,
\qquad
\I^{\imath}_{odd} = \cup^{\infty}_{r=0} \I^{\imath}_{r} = \Z_{>0},
\qquad
I_{odd}  = \Z+\hf.\\
\label{eq:II}
\I_{ev} = \cup^{\infty}_{r=0}  & \I_{2r+2} = \Z+\hf,
\qquad
\I_{ev}^{\imath} = \cup^{\infty}_{r=0} \I^{\imath}_{r} = \Z_{\ge 0}+\hf,
\qquad
I_{ev}  = \Z.
\end{align}

When it is not necessary to distinguish the even or odd cases, we shall abuse the notation and simply write $\I$, $\I^{\imath}$, $I$ (of course, they have to be consistent, i.e., all even or all odd.).

We have the natural inclusions of $\Qq$-algebras:
\begin{align*}
\cdots  \subset & \U_q(\mathfrak{sl}_{2r+1}) \subset \U_q(\mathfrak{sl}_{2r+3}) \subset \cdots ,\qquad
\cdots     \subset & \Ui_q(\mathfrak{sl}_{2r+1}) \subset \Ui_q(\mathfrak{sl}_{2r+3}) \subset \cdots,\\
 \cdots \subset & \U_q(\mathfrak{sl}_{2r+2}) \subset \U_q(\mathfrak{sl}_{2r+4}) \subset \cdots ,\qquad
\cdots    \subset & \Ui_q(\mathfrak{sl}_{2r+2}) \subset \Ui_q(\mathfrak{sl}_{2r+4}) \subset \cdots.
\end{align*}
Define the following infinite rank $\Qq$-algebras:
\begin{align*}
\U_{odd} := \bigcup^{\infty}_{r=0} \U_q(\mathfrak{sl}_{2r+2}) \quad &\text{ and } \quad  \Ui_{odd} := \bigcup^{\infty}_{r=0} \Ui_q(\mathfrak{sl}_{2r+2}),\displaybreak[0]\\
\U_{ev} := \bigcup^{\infty}_{r=0} \U_q(\mathfrak{sl}_{2r+1})   \quad &\text{ and } \quad \Ui_{ev} := \bigcup^{\infty}_{r=0} \Ui_q(\mathfrak{sl}_{2r+1}).
\end{align*}

We also abuse the notation and simply write the pair ($\U, \Ui$) (with the same subscripts).
The  embeddings of finite rank algebras 
induce an embedding of $\Qq$-algebras, denoted also by $\iota : \bun \longrightarrow \U$.
Again $\U$ is naturally a Hopf algebra with coproduct $\Delta$, and its restriction under $\iota$,
$
\Delta: \bun \rightarrow \bun \otimes \U,
$
makes $\bun$ (or more precisely $\iota(\bun)$) naturally a (right) coideal subalgebra of $\U$. 
The anti-linear bar involutions on finite rank algebras induce anti-linear bar involution $\psi$ on $\U$ and anti-linear bar involution $\psi_{\imath}$ on $\Ui$, respectively.

Recall $\Pi_{k}$ denotes the simple system of $\U_q(\mathfrak{sl}_{k})$. Let
$
\Pi_{odd} := \bigcup^{\infty}_{r=0} \Pi_{2r+1}
$ ($\Pi_{ev} := \bigcup^{\infty}_{r=0} \Pi_{2r+2}$, respectively)
be a simple system of $\U_{odd}$ ($\U_{ev}$, respectively). We again shall write $\Pi$ for both $\Pi_{odd}$ and $\Pi_{ev}$.
Recall we denote the integral weight lattice of $\U_{k}$ by $\Lambda_{k}$. Then let 
\[
\Lambda_{odd} := \oplus_{i \in I_{odd}} \Z[\varepsilon_i] = \bigcup^{\infty}_{r =0} \Lambda_{2r+1} \quad \text{ and } \quad \Lambda_{ev} := \oplus_{i \in I_{ev}} \Z[\varepsilon_i] = \bigcup^{\infty}_{r =0} \Lambda_{2r+2}
\] 
be  the integral weight lattice of $\U_{odd}$ and $\U_{ev}$, respectively. Thus by abuse of notations, we have (for both cases)
\[
\Lambda = \oplus_{i \in I} \Z[\varepsilon_i].
\]
Following \S \ref{subsec:theta}, we have the quotient lattice $\Lambda_{\inv}$ of the lattice $\Lambda$.

Following \cite[\S8.1]{BW13} we can define the intertwiner $\Upsilon$ (which lies in some completion of $\U^{-}$) for the quantum symmetric pair $(\U, \Ui)$ such that 
\[
\Upsilon := \sum_{\mu \in \N\Pi} \Upsilon_{\mu}, \quad \Upsilon_\mu \in \U^{-}_{\mu}.
\]
We shall see that $\Upsilon$ is a well-defined operator on $\U$-modules that we are concerned.

\subsection{The Lie superalgebra $\mf{osp}(2m|2n)$}
   \label{subsec:osp}
   
In this subsection, we recall some basics on ortho-symplectic Lie superalgebras and set up notations to be used later on (cf. \cite{CW12} for 
more on Lie superalgebras). 

Let $\Z_2 = \{\ov{0}, \ov{1}\}$.
 Let $\C^{2m|2n}$ be a superspace of dimension $(2m|2n)$ with basis 
 $\{e_i \mid 1 \leq i \leq 2m\} \cup \{e_{\ov j} \mid 1 \leq j \leq 2n\}$, 
 where the $\Z_2$-grading is given by the following parity function:
\[
p(e_i) = \ov 0, \qquad p(e_{\ov j}) = \ov 1 \quad (\forall i,j).
\] 
Let $B$ be a non-degenerate even supersymmetric bilinear form on $\C^{2m|2n}$.
The general linear Lie superalgebra $\mf{gl}(2m|2n)$ is the Lie superalgebra of linear transformations on $\C^{2m|2n}$
(in matrix form with respect to the above basis).
 For $s \in \Z_2$, we define
\begin{align*}
\osp(2m|2n)_s &:= \{ g \in \mf{gl}(2m|2n)_s \mid B(g(x), y) = -(-1)^{s \cdot p(x)}B(x, g(y))\},\\
\osp(2m|2n) &:=\osp(2m|2n)_{\ov 0} \oplus \osp(2m|2n)_{\ov 1}.
\end{align*}

We now give a matrix realization of the Lie superalgebra $\osp(2m|2n)$. Take the supersymmetric bilinear form $B$  
with the following matrix form, with respect to the basis $(e_1, e_2, \dots, e_{2m}, e_{\ov 1}, e_{\ov 2}, \dots, e_{\ov{2n}})$: 
\[\mc J_{2m|2n} :=
\begin{pmatrix}
0 & I_m  & 0& 0\\
I_m & 0  & 0 & 0\\
0 & 0  & 0 & I^n\\
0 & 0  & -I^n & 0
\end{pmatrix}
\]

 Let $E_{i,j}$, $1 \leq i,j \leq 2m$, and $E_{\ov k , \ov h}$, $ 1 \leq k,h \leq 2n$, be the $(i,j)$th and 
 $(\ov k, \ov h)$th elementary matrices, respectively. The Cartan subalgebra of $\osp(2m|2n)$ 
 of diagonal matrices is denoted by $\mf h_{m|n}$, which is spanned by
\begin{align*}
&H_i := E_{i,i}-E_{m+i,m+i}, \quad  1 \leq i \leq m,\\
&H_{\ov j} :=E_{\ov j, \ov j} - E_{\ov{n+j}, \ov{n+j}}, \quad 1 \leq j \leq n.
\end{align*}
We denote by $\{\ep_i, \ep_{\ov j} \mid 1 \leq i \leq m, 1 \leq j \leq n \}$ the basis of $\mf h^*_{m|n}$ such that
\[
\ep_{a}(H_b) = \delta_{a,b}, \quad \text{ for  } a, b \in \{i, \ov j \mid 1 \leq i \leq m, 1 \leq j \leq n\}.
\]
We denote the lattice of integral weights of $\osp(2m|2n)$ by 
\begin{equation}  \label{eq:Xmn}
X_{ev}(m|n) := \sum^{m}_{i=1}\Z\ep_{i} + \sum^n_{j=1}\Z\ep_{\ov j}.
\end{equation}
Denote the set of half integral weights of $\osp(2m|2n)$ by
\[
X_{odd}(m|n) := \sum^{m}_{i=1}(\Z+\hf) \ep_{i} + \sum^n_{j=1}(\Z+\hf) \ep_{\ov j}.
\]
When it is not necessary to distinguish the integral or half-integral weights we shall abuse the notation, and simply write $X (m |n)$ for both of them.

The supertrace form on $\osp(2m|2n)$ induces a non-degenerate symmetric bilinear form 
on $\mf h^*_{m|n}$ denoted by $(\cdot | \cdot)$, such that
\[
(\ep_{i}\vert\ep_{a}) = \delta_{i,a}, \quad (\ep_{\ov j}| \ep_{a}) 
 = -\delta_{\ov j, a}, \quad \text{ for  } a \in   \{i, \ov j \mid 1 \leq i \leq m, 1 \leq j \leq n\}.
\]
We have the following root system of $\osp(2m|2n)$ with respect to $\mf h_{m|n}$
\[\Phi = \Phi_{\ov 0} \cup \Phi_{\ov 1} 
= \{\pm\ep_{i}\pm\ep_{j}, \pm\ep_{\ov k}\pm\ep_{\ov l}, \pm2\ep_{\ov q}\} \cup \{\pm\ep_{p}\pm\ep_{\ov q}\},
\]
where $1 \leq i < j \leq m$, $1 \leq p \leq m$, $1\leq q \leq n$, $1 \leq k < l\leq n$. 

In this paper we shall need to deal with various Borel subalgebras, hence various simple systems of $\Phi$. 
Let ${\bf b}=(b_1,b_2,\ldots,b_{m+n})$ be a
sequence of $m+n$ integers such that $m$ of the $b_i$'s are equal to
${0}$ and $n$ of them are equal to ${1}$. We call such a sequence a
{\em $0^m1^n$-sequence}. 
Associated to each $0^m1^n$-sequence ${\bf b} =(b_1, \ldots, b_{m+n})$, 
we have the following fundamental system  $\Pi_{\bf {b}}$, and hence a positive system 
$\Phi_{\bf b}^+ =\Phi_{{\bf b},\bar{0}}^+ \cup \Phi_{{\bf b},\bar{1}}^+$, of the root system $\Phi$ of $\mathfrak{osp}(2m|2n)$:
\begin{align*}
\Pi_{\bf {b}} &= \{-\ep^{b_1}_1 - \ep^{b_2}_2, \ep^{b_i}_i -\ep^{b_{i+1}}_{i+1} \mid 1 \leq i \leq m+n-1\}, \qquad & \text{ for } b_1=0;\\
\Pi_{\bf {b}} &= \{-2\ep^{b_1}_1, \ep^{b_i}_i -\ep^{b_{i+1}}_{i+1} \mid 1 \leq i \leq m+n-1\}, \qquad & \text{ for } b_1=1.
\end{align*}
where $\ep^{0}_i = \ep_{x}$ for some $1 \leq x \leq m$, $\ep^1_{j} = \ep_{\ov y}$ 
for some $1 \leq y \leq n$, such that $\ep_{x} -\ep_{x+1}$ and $\ep_{\ov y} - \ep_{\ov{y+1}}$ 
are always positive. It is clear that $\Pi_{\bf b}$ is uniquely determined by these restrictions.
The Weyl vector associate with the fundamental system $\Pi_{{\bf b}}$ is defined to be
$\rho_{\bf b}:= \hf \sum_{\alpha \in \Phi^+_{{\bf b}, \bar{0}}} \alpha -\hf \sum_{\beta \in \Phi^+_{{\bf b}, \bar{1}}} \beta$.

Corresponding to ${\bf b}^{\text{st}} =(0,\ldots, 0,1,\ldots,1)$, we have the following standard 
Dynkin diagram associated to $\Pi_{{\bf b}^{\text{st}}}$ (for $m\ge 2$):

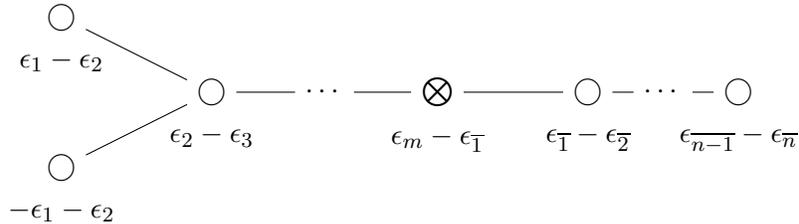
\begin{center}
\begin{tikzpicture}
\draw (-1,1) node[label=below:$\epsilon_1 - \epsilon_2$] (1) {$\bigcirc$} ;
\draw (-1,-1) node[label=below:$-\epsilon_1 - \epsilon_2$] (0) {$\bigcirc$};
\draw (1,0) node[label=below:$\epsilon_2 - \epsilon_3$] (2) {$\bigcirc$};
\draw (2.5,0) node (3) {$\cdots$};
\draw (4,0) node[label=below:$\epsilon_m - \epsilon_{\overline{1}}$] (4) {$\bigotimes$};
\draw (6,0) node[label=below:$\epsilon_{\overline{1}} - \epsilon_{\overline{2}}$] (5) {$\bigcirc$};
\draw (7,0) node (6) {$\cdots$};
\draw (8,0) node[label=below:$\epsilon_{\overline{n-1}} - \epsilon_{\overline{n}}$] (7) {$\bigcirc$};

\draw (1) -- (2);
\draw (0) -- (2);
\draw (2) -- (3);
\draw (3) -- (4);
\draw (4) -- (5);
\draw (5) -- (6);
\draw (6) -- (7);

\end{tikzpicture}
\end{center}
As usual, $\bigotimes$ stands for an isotropic simple odd root, $\bigcirc$ stands for an simple even root.

\begin{rem}
If we have $m =1$, the corresponding Dynkin diagram becomes (with $n \ge 2$):
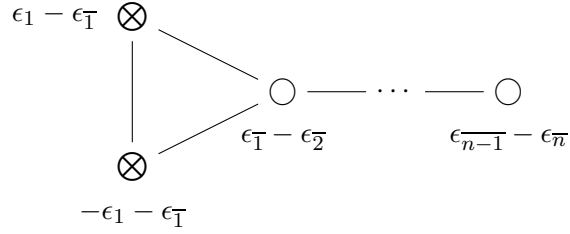
\begin{center}
\begin{tikzpicture}
\draw (-1,1) node[label=left:$\epsilon_1 - \epsilon_{\overline{1}}$] (1) {$\bigotimes$} ;
\draw (-1,-1) node[label=below:$-\epsilon_1 - \epsilon_{\overline{1}}$] (0) {$\bigotimes$};
\draw (1,0) node[label=below:$\epsilon_{\overline{1}} - \epsilon_{\overline{2}}$] (2) {$\bigcirc$};
\draw (2.5,0) node (3) {$\cdots$};
\draw (4,0) node[label=below:$\epsilon_{\overline{n-1}} - \epsilon_{\overline{n}}$] (4) {$\bigcirc$};

\draw (1) -- (2);
\draw (0) -- (2);
\draw (2) -- (3);
\draw (3) -- (4);
\draw (1) -- (0);

\end{tikzpicture}
\end{center}
\end{rem}

A direct computation shows that
\begin{equation}  \label{eq:rhobst}
\rho_{{\bf b}^{\text{st}}} = 0\epsilon_1 - \epsilon_2 -\ldots -(m-1) \epsilon_m
+ (m-1) \epsilon_{\bar{1}} 
+\ldots + (m-n) \epsilon_{\bar{n}}.
\end{equation}

More generally, associated to a sequence $\bf b$ which starts with two $0$'s is a Dynkin diagram which always starts on the left with a type $D$ branch:

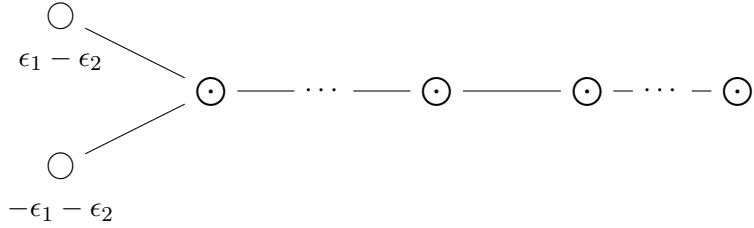
\begin{center}
\begin{tikzpicture}
\draw (-1,1) node[label=below:$\epsilon_1 - \epsilon_2$] (1) {$\bigcirc$} ;
\draw (-1,-1) node[label=below:$-\epsilon_1 - \epsilon_2$] (0) {$\bigcirc$};
\draw (1,0) node (2) {$\bigodot$};
\draw (2.5,0) node (3) {$\cdots$};
\draw (4,0) node (4) {$\bigodot$};
\draw (6,0) node (5) {$\bigodot$};
\draw (7,0) node (6) {$\cdots$};
\draw (8,0) node (7) {$\bigodot$};

\draw (1) -- (2);
\draw (0) -- (2);
\draw (2) -- (3);
\draw (3) -- (4);
\draw (4) -- (5);
\draw (5) -- (6);
\draw (6) -- (7);

\end{tikzpicture}
\end{center}
Here $\bigodot$ stands for either $\bigotimes$ or $\bigcirc$ depending on $\bf b$.

On the other hand, corresponding to ${\bf b}^{\text{st}'} =(1,\ldots, 1,0,\ldots,0)$, we have the following another often used Dynkin diagram associated to $\Pi_{{\bf b}^{\text{st}'}}$:

\begin{center}
\begin{tikzpicture}
\draw (-0.2,0) node[label=below:$-2\epsilon_{\overline{1}}$]  (0) {$\bigcirc$};
\draw (1,0) node[label=below:$\epsilon_{\overline{1}} - \epsilon_{\overline{2}}$] (2) {$\bigcirc$};
\draw (2.5,0) node (3) {$\cdots$};
\draw (4,0) node[label=below:$\epsilon_{\overline{n}} - \epsilon_{1}$]  (4) {$\bigotimes$};
\draw (6,0) node[label=below:$\epsilon_1 - \epsilon_2$]  (5) {$\bigcirc$};
\draw (7,0) node (6) {$\cdots$};
\draw (8,0) node[label=below:$\epsilon_{m-1} - \epsilon_m$]  (7) {$\bigcirc$};
\draw (0.4,0) node {$\Longrightarrow$};
\draw (2) -- (3);
\draw (3) -- (4);
\draw (4) -- (5);
\draw (5) -- (6);
\draw (6) -- (7);

\end{tikzpicture}
\end{center}
A direct computation shows that 
\begin{equation}
\rho_{{\bf b}^{\text{st}'}} = - \epsilon_{\overline{1}} - 2\epsilon_{\overline{2}} - \cdots - n \epsilon_{\overline{n}} + n\epsilon_{1} + (n-1)\epsilon_{1}+\cdots + (n-m)\epsilon_{m}.
\end{equation}
\begin{rem}
The fundamental system $\Pi_{{\bf b}_1}$ with ${\bf b}_1 = (0,1,{\bf b}')$ and the fundamental system $\Pi_{{\bf b}_2}$ with ${\bf b}_2 = (1,0,{\bf b}')$ differ by an odd reflection (\cite[Remark~1.31]{CW12}), even though their corresponding Dynkin diagrams look quite different. 
\end{rem}

More generally, associated to a sequence $\bf b$ which starts with one $1$ is a Dynkin diagram which always starts on the left with a type $C$ branch:

\begin{center}
\begin{tikzpicture}
\draw (-0.2,0) node[label=below:$-2\epsilon_{\overline{1}}$]  (0) {$\bigcirc$};
\draw (1,0) node (2) {$\bigodot$};
\draw (2.5,0) node (3) {$\cdots$};
\draw (4,0) node   (4) {$\bigodot$};
\draw (6,0) node   (5) {$\bigodot$};
\draw (7,0) node (6) {$\cdots$};
\draw (8,0) node   (7) {$\bigodot$};
\draw (0.4,0) node {$\Longrightarrow$};
\draw (2) -- (3);
\draw (3) -- (4);
\draw (4) -- (5);
\draw (5) -- (6);
\draw (6) -- (7);

\end{tikzpicture}
\end{center}

%
%
%

Now we can write the non-degenerate symmetric bilinear form on $\Phi$ as follows:
\[
(\ep^{b_i}_i | \ep^{b_j}_j) = (-1)^{b_i} \delta_{ij}, \quad \quad \quad 1 \leq i, j \leq m+n.
\] 
We define $\mathfrak{n}_{\bf b}^\pm$ to be the nilpotent subalgebra spanned by
the positive/negative root vectors in $\osp(2m|2n)$. 
Then we obtain a triangular decomposition of $\osp(2m|2n)$:
\[
\osp(2m|2n) = \mathfrak{n}_{\bf b}^+ \oplus \mathfrak{h}_{m|n} \oplus \mathfrak{n}_{\bf b}^-,
\]
with $\mathfrak{n}_{\bf b}^+ \oplus \mathfrak{h}_{m|n}$ as a Borel subalgebra. 

Fix a $0^m1^n$-sequence ${\bf b}$ and hence a positve system $\Phi^+_{\bf b}$. 
We denote by $Z(\osp(2m|2n))$ the center of the enveloping algebra $U(\osp(2m|2n))$. 
There exists a standard projection $\phi: U(\osp(2m|2n)) \rightarrow U(\mf h_{m|n})$ which is
consistent with the PBW basis associated to the above triangular decomposition 
(\cite[\S 2.2.3]{CW12}). For $\lambda \in \mf h^*_{m|n}$, we define the central character $\chi_{\lambda}$ by letting 
$$
\chi_\lambda(z) :=\lambda(\phi(z)),\quad  \text{ for }z \in Z(\osp(2m|2n)).
$$
Denote the Weyl group of (the even subalgebra of) $\osp(2m|2n)$ by $W_{\osp}$,
which is isomorphic to $W_{D_m} \times W_{C_n}$. 
Then for $\mu$, $\nu \in \mf h^*_{m|n}$, we say $\mu$, $\nu$ are linked and denote it by $\mu \sim  \nu$, 
if there exist mutually orthogonal isotropic odd roots $\alpha_1, \alpha_2, \dots, \alpha_l$, 
complex numbers $c_1, c_2, \dots, c_l$, and an element $w \in W_{\osp}$ satisfying 
\[
\mu + \rho_{\bf b}= w(\nu +\rho_{\bf b} - \sum_{i=1}^{l}c_i\alpha_i), \quad (\nu + \rho_{\bf b}|  \alpha_j)= 0, \quad j=1 \dots, l.
\]
It is clear that $\sim$ is an equivalent relation on $\mf h^*_{m|n}$.  Versions of the following basic fact go back
to Kac, Sergeev, and others.
\begin{prop}
  \cite[Theorem 2.30]{CW12}
Let $\lambda$, $\mu \in \mf h^*_{m|n}$. Then $\lambda$ is linked to $\mu$ if and only if $\chi_{\lambda} = \chi_{\mu}$.
\end{prop}
\subsection{The BGG categories}
\label{subsec:cat}

In this subsection, we shall define various (parabolic) BGG categories for ortho-symplectic Lie superalgebras.

\begin{definition}
 Let ${\bf b}$ be a $0^m1^n$-sequence. The Bernstein-Gelfand-Gelfand
(BGG) category $\mathcal{O}_{{\bf b}}$ ($= \mathcal{O}_{{\bf b},ev}$ or $\mathcal{O}_{{\bf b},odd}$, respectively) 
is the category of  
$\h_{m|n}$-semisimple $\mathfrak{osp}(2m|2n)$-modules $M$ such that
\begin{itemize}
\item[(i)]
$M=\bigoplus_{\mu\in X(m|n)}M_\mu$ and $\dim M_\mu<\infty$; (for $X(m|n) = X_{ev}(m|n)$ or $X(m|n) =X_{odd}(m|n)$, respectively)

\item[(ii)]
there exist finitely many weights ${}^1\la,{}^2\la,\ldots,{}^k\la\in X(m|n)$
(depending on $M$) such that if $\mu$ is a weight in $M$, then
$\mu\in{{}^i\la}-\sum_{\alpha\in{\Pi_{\bf b}}}\N \alpha$, for
some $i$.
\end{itemize}
The morphisms in $\mathcal{O}_{\bf b}$ are all (not necessarily even)
homomorphisms of $\mathfrak{osp}(2m|2n)$-modules.
\end{definition}

If the $0^m1^n$-sequence  ${\bf b}$ starts with $0$, then we are interested in both $ \mathcal{O}_{{\bf b},ev}$ and $\mathcal{O}_{{\bf b},odd}$. If the $0^m1^n$-sequence  ${\bf b}$ starts with $1$, then we are interested in only $ \mathcal{O}_{{\bf b},ev}$. 

Similar to \cite[Proposition 6.4]{CLW15}, all these categories $\mc O_{\bf b}$ are identical for various $\bf b$,
since the even subalgebras of the Borel subalgebras 
$\mathfrak{n}_{\bf b}^+ \oplus \mathfrak{h}_{m|n}$ are identical and the odd parts of these Borels 
always act locally nilpotently.

Denote by $M_{\bf b}(\lambda)$ the {\bf b}-Verma modules with highest weight $\lambda$. 
Denote by  $L_{\bf b}(\lambda)$ the unique simple quotient of $M_{\bf b}(\lambda)$. They are both in $\mathcal{O}_{\bf b}$.

It is well known that the Lie superalgebra $\gl(2m|2n)$ has an automorphism $\tau$ given by the formula:
\[
\tau(E_{ij}):=-(-1)^{p(i)(p(i)+p(j))}E_{ji}.
\]
The restriction of $\tau$ on $\osp(2m|2n)$ gives an automorphism of $\osp(2m|2n)$. For an object 
$M = \oplus_{\mu \in X(m|n)}M_\mu \in \mathcal{O}_{\bf b}$, we let
\[
M^{\vee}:=\oplus_{\mu \in X(m|n)}M^*_{\mu}
\]
be the restrictd dual of $M$. We define the action of $\osp(2m|2n)$ on $M^{\vee}$ by $(g \cdot f)(x) := -f(\tau(g)\cdot x)$, 
for $ f \in M^\vee, g\in \osp(2m|2n)$, and $ x\in M$. We denote the resulting module by $M^\tau$. 

An object $M \in \mathcal{O}_{\bf b}$ is said to have a ${\bf b}$-Verma flag 
(respectively, dual ${\bf b}$-Verma flag), if $M$ has a filtration 
$
0=M_0 \subseteq M_1 \subseteq M_2 \subseteq \dots \subseteq M_t = M,
$
such that $M_i/M_{i-1} \cong M_{\bf b}(\gamma_i), 1 \leq i \leq t$ (respectively, $M_i/M_{i-1} \cong M^\tau_{\bf b}(\gamma_i)$) for some $\gamma_i \in X(m|n)$.


Associated to each $\lambda \in X(m|n)$, a ${\bf b}$-tilting module $T_{\bf b}(\lambda)$ is an indecomposable 
$\osp(2m|2n)$-module in $\mathcal{O}_{\bf b}$ characterized by the following two conditions:
$T_{\bf b}(\la)$ has a ${\bf b}$-Verma flag with $M_{\bf b}(\la)$ at
the bottom;
$\text{Ext}^1_{\CatO_{\bf b}}(M_{\bf b}(\mu),T_{\bf b}(\la))=0$, for all
$\mu\in X(m|n)$.

\subsection{Fock spaces and their completions}
Let $\VV := \sum_{a \in I} \Qq v_a$ be the natural representation of $\U$, 
where the action of $\U$ on $\VV$ is defined as follows (for $i \in \I$, $a \in I$):
\[
E_{\alpha_i} v_a = \delta_{i+\hf, a} v_{a-1}, \quad F_{\alpha_i}v_a 
= \delta_{i-\hf, a}v_{a+1}, \quad K_{\alpha_i} v_a = q^{(\alpha_i, \varepsilon_a)}v_a.
\]
Let $\WW :=\VV^*$ be the restricted dual module of $\VV$ with basis $\{w_a \mid a \in I\}$ such that 
$\langle w_a, v_b \rangle =  (-q)^{-a} \delta_{a,b}$. The action of $\U$ on $\WW$ is given by the following formulas 
(for $i \in \I$, $a \in I$):
\[
E_{\alpha_i} w_a = \delta_{i-\hf, a} w_{a+1}, \quad F_{\alpha_i}w_a 
= \delta_{i+\hf, a}w_{a-1}, \quad K_{\alpha_i} w_a = q^{-(\alpha_i, \varepsilon_a)}w_a.
\]
By restriction through the embedding $\iota$, $\VV$ and $\WW$ are naturally  $\bun$-modules.

Fix a ${0^m1^n}$-sequence ${\bf b} =(b_1,b_2,\ldots,b_{m+n})$. We have the following
tensor space over $\Q(q)$, called the {\em $\bf b$-Fock space} or simply
{\em Fock space}:
\begin{equation}  \label{eq:Fock}
{\mathbb T}^{\bf b} :={\mathbb V}^{b_1}\otimes {\mathbb
V}^{b_2}\otimes\cdots \otimes{\mathbb V}^{b_{m+n}},
\end{equation}
where we denote
$${\mathbb V}^{b_i}:=\begin{cases}
{\mathbb V}, &\text{ if }b_i={0},\\
{\mathbb W}, &\text{ if }b_i={1}.
\end{cases}$$
The tensors here and in similar settings later on are understood
to be over the field $\Q(q)$.
Note that both algebras $\U$ and $\bun$ act on $\mathbb T^{\bf b}$ via an iterated coproduct.

For $f\in I^{m+n}$, we define
\begin{equation}  \label{eq:Mf}
M^{\bf b}_f :=\texttt{v}^{b_1}_{f(1)}\otimes
\texttt{v}^{b_2}_{f(2)}\otimes\cdots\otimes
\texttt{v}^{b_{m+n}}_{f(m+n)},
\end{equation}
where we use the notation
$\texttt{v}^{b_i}:=
\begin{cases}v,\text{ if }b_i={0},
\\
w,\text{ if }b_i={1}.
\end{cases}$ 
We refer to $\{M^{\bf b}_f \mid f\in
I^{m+n}\}$ as the {\em standard monomial basis} of ${\mathbb
T}^{\bf b}$.

%

Let ${\bf b} =(b_1, \cdots, b_{m+n})$ be an arbitrary $0^m1^n$-sequence.
We first define a partial ordering on $I^{m+n}$, which depends on the sequence ${\bf b}$.
There is a natural bijection $I^{m+n} \leftrightarrow X(m|n)$ (recall $X(m|n)$ from  \eqref{eq:Xmn}), defined as
\begin{align}
&f \mapsto \lambda^{\bf b}_f, \text{ where }  \lambda^{\bf b}_f 
= \sum_{i=1}^{m+n}(-1)^{b_i}f(i)\ep^{b_i}_i -\rho_{\bf b},  \quad \text{for } f \in I^{m+n}, 
\label{osp:eq:ftolambda}
\\
& \lambda \mapsto f^{\bf b}_{\lambda},
 \text{ where } f(i)=(\lambda + \rho_{\bf b} \vert \ep^{b_i}_i), \quad \quad \quad \text{for } \lambda \in X(m|n).
\label{osp:eq:lambdatof}
\end{align}

\begin{definition}Fix a ${\bf b} = (b_1, \dots, b_{m+n})$. 
For any $f\in I^{m+n}$, let $\varepsilon_f = \sum^{m+n}_{i=1} (-1)^{b_i}\varepsilon_{f(i)} \in \Lambda$. Let $\overline{\varepsilon_f}$ be the image of $\varepsilon_f$ in the quotient $\Lambda_{\inv}$. Define the (${\bf b}$-)Bruhat ordering on the set $I^{m+n}$ (hence on $X(m|n)$) as follows:
for $f, g \in I^{m+n}$, we say $g \preceq_{\bf b} f$ if $\overline{\varepsilon_{f}} = \overline{\varepsilon_{g}}$ and 
\[
\lambda^{\bf b}_f - \lambda^{\bf b}_g = a_0 (-\epsilon^{b_1}_1) + \sum^{m+n}_{i=1}a_i (\epsilon^{b_i}_i - \epsilon^{b_{i+1}}_{i+1}), \quad \text{ for } a_i \in \N.
\]
\end{definition}

\begin{rem}
This is exactly the same partial ordering used in \cite[\S8.4]{BW13}.
\end{rem}
%
%
%
%

Let the $B$-completion $\widehat{\mathbb{T}}^{\bf b}$ be the space spanned by elements of the form (possibly infinitely many non-zero $c_{gf}^{\bf b}$) 
\begin{align}
M_f+\sum_{g\prec_{\bf b}f}c_{gf}^{\bf b}(q) M_g, \quad \text{ for } c_{gf}^{\bf b}(q) \in\Q(q).
\end{align}

We know from \cite[Lemma~9.8]{BW13} that $\ibar = \Upsilon \circ \psi : \widehat{\mathbb{T}}^{\bf b} \rightarrow \widehat{\mathbb{T}}^{\bf b}$ is an anti-linear involution such that 
\[\ibar(M_f) = M_f + \sum_{g \prec_{\bf b} f} r_{gf}(q)M_g, \quad \text{ for } r_{gf}(q) \in \mA.
\]
\begin{thm} 
 \label{thm:iCBb}
The $\Qq$-vector space $\widehat{\mathbb{T}}^{\bf b}$ has unique $\Bbar$-invariant topological bases
\[
\{T^{\bf b}_f \mid f \in I^{m+n}\} \text{ and } \{L^{\bf b}_f \mid f \in I^{m+n}\}
\]
such that 
\[
T^{\bf b}_f = M_f + \sum_{g \preceq_{\bf b} f }t^{\bf b}_{gf}(q)M^{\bf b}_g, 
\quad 
L^{\bf b}_f  = M_f + \sum_{g \preceq_{\bf b} f }\ell^{\bf b}_{gf}(q)M^{\bf b}_g,
\]
with $t^{\bf b}_{gf}(q) \in q\Z[q]$, and $\ell^{\bf b}_{gf}(q) \in q^{-1}\Z[q^{-1}]$, for $g \preceq_{\bf b}f $. 
(We shall write $t^{\bf b}_{ff}(q) = \ell^{\bf b}_{ff}(q) = 1$, $t^{\bf b}_{gf}(q)=\ell^{\bf b}_{gf}(q)=0$ for $ g \not\preceq_{\bf b} f$.)
\end{thm}

\begin{definition}
$\{T^{\bf b}_f \mid f \in I^{m+n}\} \text{ and } \{L^{\bf b}_f \mid f \in I^{m+n}\}$ are call the 
{\em $\imath$-canonical basis} and {\em dual $\imath$-canonical basis} of $\widehat{\mathbb{T}}^{\bf b}$, respectively. 
The polynomials $t^{\bf b}_{gf}(q)$ and $\ell^{\bf b}_{gf}(q)$ are called {\em $\imath$-Kazhdan-Lusztig (or $\imath$-KL) polynomials}.
\end{definition}
\begin{rem}
The different $\ibar$, coming from different $\Upsilon$, leads to different (dual) $\imath$-canonical basis on the tensor space $\mathbb{T}^{\bf b}$ than the ones in \cite[Definition~9.10]{BW13}.
\end{rem}
The following theorem is a counterpart of \cite[Theorem~9.11]{BW13}.

\begin{thm}
	\begin{enumerate}
	\item	(Positivity)We have $t^{\bf b}_{gf} \in \N[q]$.
	\item	The sum $T^{\bf b}_f = M_f + \sum_{g \preceq_{\bf b} f }t^{\bf b}_{gf}(q)M^{\bf b}_g$ is finite for all $f \in I^{m+n}$.
	\end{enumerate}
\end{thm}



\subsection{Translation functors}

In \cite{Br03}, Brundan established a $\U$-module
isomorphism between the Grothendieck group of the category $\mc O$ of $\gl(m|n)$ 
and a Fock space (at $q=1$), where some properly defined translation functors acting as Chevalley generators of $\U$ at $q=1$. 
In \cite{BW13}, the analogue in the setting of $\osp(2m+1|2n)$ has been developed. Here we generalize the construction to the setting of the Lie superalgebra $\osp(2m|2n)$.

\par Let $V$ be the natural $\mathfrak{osp}(2m|2n)$-module. Notice that $V$ is self-dual. 
Recalling \S\ref{subsec:osp}, we have the following decomposition of $\mathcal{O}_{\bf b}$ (for fixed ${\bf b}$):
$$
\mathcal{O}_{\bf b} = \displaystyle\bigoplus _{\chi_\lambda} \mathcal{O}_{{\bf b}, \chi_{\lambda}},
$$ 
where $\chi_\lambda$ runs over all integral or half-integral central characters, i.e. $\lambda$ runs over the equivalent classes $X(m|n) / \sim$ (recall this means $X_{ev}(m|n) / \sim$ or $X_{odd}(m|n) / \sim$, respectively).

We write $\mathcal{O}_{{\bf b}, \gamma} := \oplus_{\chi_\lambda}  \mathcal{O}_{{\bf b}, \chi_{\lambda}}$ for all $\lambda$ such that $\overline{\varepsilon_{f^{\bf b}_{\lambda}} }= \gamma \in \Lambda_{\inv}$.
For $r \geq 0$, let $S^r V$ be the $r$th supersymmetric power of $V$. 
For $i \in \I^{\imath}$, $M \in \mathcal{O}_{{\bf b}, \gamma}$, 
we define the following translation functors in $\mathcal{O}_{\bf b}$: 
\begin{align}
\bff^{(r)}_{\alpha_{i}} M &:= \text{pr}_{\gamma - r(\varepsilon_{i-\hf}-\varepsilon_{i+\hf})}(M \otimes S^rV),
\\
\be^{(r)}_{\alpha_{i}} M &:= \text{pr}_{\gamma + r(\varepsilon_{i-\hf}-\varepsilon_{i+\hf})}(M \otimes S^rV),
\label{eq:eft}
\\
\bt M &:=\text{pr}_{\gamma}(M \otimes V), \qquad \text{ (for the case $\Ui_{odd}$)},
\end{align}
where $\text{pr}_{\mu}$ is the natural projection from $\mathcal{O}_{\bf b}$ to $\mathcal{O}_{{\bf b}, \mu}$ for $\mu  \in \Lambda_{\inv}$.

Note that the (exact) translation functors naturally induce operators on the Grothendieck group $[\CatO^{\Delta}_{\bf b}]$, 
denoted by $\bff^{(r)}_{\alpha_{i}}$, $\be^{(r)}_{\alpha_{i}}$, and $\bt$ as well.
The following two lemmas are analoges of \cite[Lemmas 4.23 and 4.24]{Br03}. 
Since they are standard, we shall skip the proofs.
\begin{lem}
On the category $\mathcal{O}_{\bf b}$, the translation functors $\bff^{(r)}_{\alpha_{i}}$, $\be^{(r)}_{\alpha_{i}}$, 
and $\bt$ are all exact. They commute with the $\tau$-duality.
\end{lem}

\begin{lem}
  \label{osp:lem:MOSV}
Let $\nu_1$, $\dots$, $\nu_N$ be the set of weights of $S^rV$ ordered so that $v_i > v_j$ if and only if $ i <j$. 
Let $\lambda \in X(m|n)$. Then $M_{\bf b}(\lambda) \otimes S^rV$ has a multiplicity-free Verma flag with 
subquotients isomorphic to $M_{\bf b}(\lambda + \nu_1)$, $\dots$, $M_{\bf b}(\lambda + \nu_N)$ in the order from bottom to top.
\end{lem}

Let $\Tb_\mA$ be the $\mA$-lattice spanned by the standard monomial basis of the $\Qq$-vector space $\mathbb{T}^{\bf b}$. 
We define $\mathbb{T}^{\bf b}_{\Z} =\Z \otimes_\mA \Tb_{\mA}$
where $\mA$ acts on $\Z$ with $q=1$. For any $u$ in the $\mA$-lattice $\Tb_\mA$, 
we denote by $u(1)$ its image in $\mathbb{T}^{\bf b}_{\Z}$.

Let $\mathcal{O}^{\Delta}_{\bf b}$ be the full subcategory of $\mathcal{O}_{\bf b}$ consisting of all modules possessing 
a finite ${\bf b}$-Verma flag. Let $\left[\mathcal{O}^{\Delta}_{\bf b}\right]$ be its Grothendieck group. 
The following lemma is immediate from the bijection $I \leftrightarrow {}X(m|n)$ (with consistent choice of the subscript, i.e., both $ev$ or $odd$).

\begin{lem} 
  \label{int:lem:OtoT}
The map
\[
\Psi : \left[\mathcal{O}^{\Delta}_{\bf b}\right] \longrightarrow \mathbb{T}_{\Z}^{\bf b}, 
\quad \quad \quad [M_{\bf b}(\lambda)] \mapsto M^{\bf b}_{f^{\bf b}_{\lambda}}(1),
\]
defines an isomorphism of $\Z$-modules. 
\end{lem}

Denote by ${_\Z \U} = \Z \otimes_{\mA}\,  {_{\mA}\U}$ the specialization of the $\mA$-algebra ${_{\mA}\U}$ at $q=1$. Hence we can view $\mathbb{T}_{\Z}^{\bf b}$ as a ${_\Z \U} $-module. 
Thanks to \eqref{int:eq:beZ}, \eqref{int:eq:bffZ}, \eqref{eq:beZ} and \eqref{eq:bffZ}, we know $\iota(\bff^{(r)}_{\alpha_{i}})$ 
and $\iota(\be^{(r)}_{\alpha_{i}})$ lie in ${_{\mA}\U}$, hence their specializations at $q=1$ in ${_\Z \U} $ act on $\mathbb{T}_{\Z}^{\bf b}$. The following proposition is a counterpart of \cite[Proposition~11.9]{BW13}.

\begin{prop}
 \label{prop:translation}
Under the identification $[\CatO^{\Delta}_{\bf b}]$ and $\mathbb{T}_{\Z}^{\bf b}$ via the isomorphism $\Psi$, 
the translation functors $\bff^{(r)}_{\alpha_{i}}$, $\be^{(r)}_{\alpha_{i}}$, and $\bt$ act in the same way 
as the specialization of  $\bff^{(r)}_{\alpha_{i}}$, $\be^{(r)}_{\alpha_{i}}$, and $\bt$ in $\bun$. 
\end{prop}


\subsection{$\imath$-Kazhdan-Lusztig theory for $\osp(2m|2n)$}
%
%

We define $\left[\left[ \mathcal{O}^{\Delta}_{\bf b} \right]\right]$ as the completion of 
$\left[\mathcal{O}^{\Delta}_{\bf b}\right]$ such that the extension of $\Psi$
\begin{equation*}
\Psi: \left[\left[ \mathcal{O}^{\Delta}_{\bf b} \right]\right] \longrightarrow \widehat{\mathbb{T}}_\Z^{\bf b} 
\end{equation*}
is an isomorphism of $\Z$-modules.
\begin{thm}
  
  \begin{enumerate}
	\item	For any $0^m1^n$-sequence ${\bf b}$ starting with $0^2$, the isomorphism 
$\Psi : \left[ \left[\CatO^{\Delta}_{\bf b}\right]\right] \rightarrow \widehat{\mathbb{T}}_{\Z}^{\bf b}$ 
satisfies 
\[
\Psi([L_{{\bf b}}(\lambda)]) = L^{{\bf b}}_{f^{{\bf b}}_{\lambda}}(1), \quad \quad \quad \Psi([T_{{\bf b}}(\lambda)]) 
= T^{{\bf b}}_{f^{{\bf b}}_{\lambda}}(1), \quad \quad \text{ for } \lambda \in {}X(m|n).
\]
	\item	For any $0^m1^n$-sequence ${\bf b}$ starting with $1$, the isomorphism 
$\Psi : \left[ \left[\CatO^{\Delta}_{{\bf b},ev}\right]\right] \rightarrow \widehat{\mathbb{T}}_{{\Z},ev}^{\bf b}$ 
satisfies 
\[
\Psi([L_{{\bf b}}(\lambda)]) = L^{{\bf b}}_{f^{{\bf b}}_{\lambda}}(1), \quad \quad \quad \Psi([T_{{\bf b}}(\lambda)]) 
= T^{{\bf b}}_{f^{{\bf b}}_{\lambda}}(1), \quad \quad \text{ for } \lambda \in {}X_{ev}(m|n).
\]

\end{enumerate}
\end{thm}

\begin{proof}This proof is essentially the same induction as the one in \cite[Theorem~11.13]{BW13} (or its predecessor \cite{CLW15}). The setting on the Fock spaces is exactly the same as \cite{BW13}, since the only difference is the precise formula of the bar involution $\ibar$. (Recall we are using the same partial ordering as in \cite{BW13}.)
Hence here we will be contented with specifying how each step follows
and refer the reader to the proof of \cite[Theorem~11.13]{BW13} (and the references therein) for details.

The inductive procedure case (1), denoted by
 $ \imath\texttt{KL}(m|n) \, \forall m \ge 2 \Longrightarrow \imath\texttt{KL}(m|n+1)$,
 is divided into the following steps:
{\allowdisplaybreaks
\begin{align}
 \imath\texttt{KL}(m+k|n) \;\; \forall k
 & \Longrightarrow  \imath\texttt{KL}(m|n|k) \;\; \forall k, \text{ by changing Borels}
  \label{ind:oddref} \\
 &\Longrightarrow \imath\texttt{KL}(m|n|\underline{k})
  \;\; \forall k,  \text{ by passing to parabolic}
  \label{ind:para} \\
 &\Longrightarrow \imath\texttt{KL}(m|n|\underline{\infty}),  \text{ by taking $k\mapsto \infty$}
  \label{ind:infty} \\
 & \Longrightarrow \imath\texttt{KL}(m|n+\underline{\infty}), \text{ by super duality}
  \label{ind:SD}  \\
 & \Longrightarrow \imath\texttt{KL}(m|n+1)\;\; \forall m, \text{ by truncation}.
  \label{ind:trunc}
\end{align}
 }
It is instructive to write down the Fock spaces corresponding to the
steps above:
{\allowdisplaybreaks
\begin{align*}
 \VV^{\otimes (m+k)} \otimes \WW^{\otimes n}\;\;  \forall k
 &\Longrightarrow  \VV^{\otimes m} \otimes \WW^{\otimes n} \otimes
 \VV^{\otimes k}\;\;  \forall k
  \\
 &\Longrightarrow \VV^{\otimes m} \otimes \WW^{\otimes n}\otimes
 \wedge^k\VV\;\; \forall k
   \\
 &\Longrightarrow \VV^{\otimes m} \otimes \WW^{\otimes n}\otimes
 \wedge^\infty \VV
   \\
 & \Longrightarrow \VV^{\otimes m} \otimes \WW^{\otimes n}\otimes
 \wedge^\infty\WW
   \\
 & \Longrightarrow
 \VV^{\otimes m} \otimes \WW^{\otimes (n+1)}\;\;\; \forall m \ge 2.
\end{align*}
}
Thanks to Theorem~\ref{thm:samebar} and Corollary~\ref{cor:samebar}, the base case for the
induction, $ \imath\texttt{KL}(m|0)$, is equivalent to the original
Kazhdan-Lusztig conjecture \cite{KL} for $\mf{so}(2m)$.

Step~\eqref{ind:oddref} follows from \cite[Proposition~11.14]{BW13}.

Step~\eqref{ind:para} follows from \cite[\S11.2]{BW13}.

Step~\eqref{ind:infty} follows from \cite[Proposition~11.4]{BW13}.

Step~\eqref{ind:SD}  is based on \cite[Proposition~11.12]{BW13}.

Step~\eqref{ind:trunc} is based on  \cite[Propositions~7.7, 11.4 and 9.17]{BW13}. 

The inductive procedure for case (2), denoted by
 $ \imath\texttt{KL}(n|m) \, \forall n \ge 1 \Longrightarrow \imath\texttt{KL}(n|m+1)$,
 is divided into the following steps:
{\allowdisplaybreaks
\begin{align}
 \imath\texttt{KL}(n+k|m) \;\; \forall k
 & \Longrightarrow  \imath\texttt{KL}(n|m|k) \;\; \forall k, \text{ by changing Borels}
  \label{C:oddref} \\
 &\Longrightarrow \imath\texttt{KL}(n|m|\underline{k})
  \;\; \forall k,  \text{ by passing to parabolic}
  \label{C:para} \\
 &\Longrightarrow \imath\texttt{KL}(n|m|\underline{\infty}),  \text{ by taking $k\mapsto \infty$}
  \label{C:infty} \\
 & \Longrightarrow \imath\texttt{KL}(n|m+\underline{\infty}), \text{ by super duality}
  \label{C:SD}  \\
 & \Longrightarrow \imath\texttt{KL}(n|m+1)\;\; \forall n, \text{ by truncation}.
  \label{C:trunc}
\end{align}
 }
The Fock spaces corresponding to the
steps above are the following:
{\allowdisplaybreaks
\begin{align*}
 \WW^{\otimes (n+k)} \otimes \VV^{\otimes m}\;\;  \forall k
 &\Longrightarrow  \WW^{\otimes n} \otimes \VV^{\otimes m} \otimes
 \WW^{\otimes k}\;\;  \forall k
  \\
 &\Longrightarrow \WW^{\otimes n} \otimes \VV^{\otimes m}\otimes
 \wedge^k\WW\;\; \forall k
   \\
 &\Longrightarrow \WW^{\otimes n} \otimes \VV^{\otimes m}\otimes
 \wedge^\infty \WW
   \\
 & \Longrightarrow \WW^{\otimes n} \otimes \VV^{\otimes m}\otimes
 \wedge^\infty\VV
   \\
 & \Longrightarrow
 \WW^{\otimes n} \otimes \VV^{\otimes (m+1)}\;\;\; \forall n \ge 1.
\end{align*}
}
Thanks to Theorem~\ref{thm:KLC}, the base case for the
induction, $ \imath\texttt{KL}(n|0)$, is equivalent to the original
Kazhdan-Lusztig conjecture \cite{KL} for $\mf{sp}(2n)$. The rest of the proof is similar to the previous case.

The theorem is proved.
\end{proof}

\begin{rem}
There is a similar Fock space formulation for various parabolic subcategories of $\osp(2m|2n)$-modules.
\end{rem}



\end{document}